\documentclass[reprint]{elsarticle}
\usepackage{amsmath,amssymb,graphicx, amsthm,color}

\newtheorem{Thm}{Theorem}

\newtheorem{Lem}{Lemma}
\newtheorem{Rmk}{Remark}

\begin{document}

\begin{frontmatter}
\journal{  }

\title{High-order algorithms for solving eigenproblems over discrete surfaces}

 \author[Chen]{Sheng-Gwo Chen\corref{corresponding}}
 \author[Wu]{Mei-Hsiu Chi}\ead{mhchi@math.ccu.edu.tw}
 \author[Wu]{Jyh-Yang Wu}\ead{jywu@math.ccu.edu.tw}

\cortext[corresponding]{Corresponding author email :
csg@mail.ncyu.edu.tw}
\address[Chen]{Department of Applied Mathematics, National Chiayi University,  Chia-Yi 600,
Taiwan.}
\address[Wu]{Department of Mathematics, National Chung Cheng University, Chia-Yi 621,
Taiwan.}

\begin{abstract}
The eigenvalue problem of the Laplace-Beltrami operators on curved
surfaces plays an essential role in the convergence analysis of the
numerical simulations of some important geometric partial
differential equations which involve this operator. In this note we
shall combine the local tangential lifting (LTL) method  with the
configuration equation to develop a new effective and convergent
algorithm to solve the eigenvalue problems of the Laplace-Beltrami
operators acting on functions over discrete surfaces. The
convergence rates of our algorithms of discrete Laplace-Beltrami
operators over surfaces is $O(r^n)$, $n \geq 1$, where $r$
represents the size of the mesh of discretization of the surface.
The problem of high-order accuracies will also be discussed and used
to compute geometric invariants of the underlying surfaces. Some
convergence tests and eigenvalue computations on the sphere, tori
and a dumbbell are presented.

\end{abstract}
\begin{keyword}
Eigenproblem, Local tangential lifting method,  Configuration
equation, Discrete Laplace-Beltrami operator.
\end{keyword}

\end{frontmatter}

\section{Introduction}
Let $\Sigma$ be a smooth regular surface in the 3D space. The
Laplace-Beltrami (LB) operator is a natural generalization of the
classical Laplacian $\Delta_{\Sigma}$ from the Euclidean space to a
curved space. To understand the LB operators on curved surfaces, it
is natural to investigate their associated eigenvalue problems:

\begin{equation}\label{problem_LB_1}
\Delta_{\Sigma} \phi = \lambda \phi.
\end{equation}
Or, more generally,
\begin{equation}\label{problem_LB_2}
\nabla_{\Sigma}\cdot (h \nabla_{\Sigma} \phi) = \lambda \phi
\end{equation}
where $h$ is $C^2$ real function on $\Sigma$.

The eigenvalue problem of the LB operator plays important roles not
just in the study of geometric properties of curved spaces, but also
in many applications in the fields of physics, engineering and
computer science. The LB operator has recently many applications in
a variety of different areas, such as surface
processing\cite{Clarenz,Sapiro}, signal processing\cite{Romeny} and
geometric partial differential equations\cite{Desbrun}.

Since the objective underlying surfaces to be considered are usually
represented as discrete meshes in these applications,  it is useful
in practice to discretize the LB operators and solving the
eigenproblems over discrete surfaces.  There are many approaches for
estimating Laplace-Beltrami operator  and solving the
Laplace-Beltrami eigenproblems \cite{Shi1,Shi2}. In 2011,
Macdonald\cite{Macdonald} proposed an elegant method to solve the
Laplace-Bletrami eigenproblems for Equations (\ref{problem_LB_1})
and (\ref{problem_LB_2}) by the closest point method\cite{Ruuth}. In
this paper we shall describe  simple and effective methods with
high-order accuracies to define the discrete LB operator on
functions on a triangular mesh.

In 2012, Ray et al.\cite{Ray} used the method of least square to
obtain high-order approximations of derivatives and integrations. In
this paper, we shall use  ideas  developed in Chen, Chi and
Wu\cite{Chen3,Chen4} where we try to estimate the discrete partial
derivatives of functions on 2D scattered data points. Indeed, the
ideas that we shall use to develop our algorithms are divided into
two main steps: first we lift the 1-neighborhood points to the
approximating tangent space and obtain a local tangential polygon.
Second, we use some geometric idea to lift functions to the tangent
space. We call this a local tangential lifting (LTL)
method\cite{Chen5,Wu}. Then we present a new algorithm, the
configuration method, to compute their Laplacians in the 2D tangent
space. This means that the LTL process allows us to reduce the 2D
curved surface problem to the 2D Euclidean problem.

In other words, we shall combine the local tangential lifting (LTL)
method  with the configuration equation to develop a new effective
and convergent algorithm to solve the eigenpair problems of the
Laplace-Beltrami operators acting on functions over curved surfaces.
We shall also present a mathematical proof of the convergence of our
algorithm. Our algorithm is not only conceptually simple, but also
easy to implement. Indeed, the convergence rate of our new
algorithms of discrete Laplace-Beltrami operators over surfaces is
$O(r^n)$, $n \geq 1$, where $r$ represents the size of the mesh of
discretization of the surface. In section 2, we introduce the
gradient of a function, the divergence of a vector field and the
Laplace-Beltrami operator on regular surfaces. Our $O(r)$-LTL
configuration method is discussed in section 3. In section 4, we
discuss how to improve the methods to have high-order accuracies. We
also give some numerical simulations to support these results in
section 5.

\section{The gradient, divergence, and   Laplace-Beltrami operator}

In order to describe the gradient, divergence and the LB operator on
functions or vector fields in a regular surface $\Sigma$ in the 3D
Euclidean space $\mathbb{R}^3$, we consider a parameterization
$\mathbf{x} : U \rightarrow \Sigma$ at a point $\mathbf{p}$, where
$U$ is an open subset of the 2D Euclidean space $\mathbb{R}^2$. We
can choose, at each point $\mathbf{q}$ of $\mathbf{x}(U)$, a unit
normal vector $\mathbf{N}(\mathbf{q})$. The map $\mathbf{N}:
\mathbf{x}(U) \rightarrow \mathbb{S}^2$ is the local Gauss map from
an open subset of the regular surface $\Sigma$ to the unit sphere
$\mathbb{S}^2$ in the 3D Euclidean space $\mathbb{R}^3$. Denote the
tangent space of $\Sigma$ at the point $\mathbf{p}$ by
$T\Sigma_{\mathbf{p}} = \{ \mathbf{v} \in \mathbb{R}^3 | \mathbf{v}
\bot \mathbf{N}(\mathbf{p}) \}$. The tangent space
$T\Sigma_{\mathbf{p}}$ is a linear space spanned by $\{ \mathbf{x}_u
, \mathbf{x}_v \}$ where $u, v$ are coordinates for $U$.

The gradient $\nabla_{\Sigma} g$  of a smooth function  $g$ on
$\Sigma$ can be computed from

\begin{equation}
\nabla_{\Sigma} g = \frac{g_u G - g_v F}{EG-F^2}\mathbf{x}_u + \frac{g_v E -g_u F}{EG-F^2}\mathbf{x}_v
\end{equation}
where $E, F$, and $G$ are the coefficients of the first fundamental
form and
\begin{equation}
g_u = \frac{\partial g(\mathbf{x}(u,v))}{\partial u} \mbox{ and }
g_v = \frac{\partial g(\mathbf{x}(u,v))}{\partial v}.
\end{equation}
See do Carmo\cite{Docarmo,Docarmo2}  for the details.

Let $\mathbf{X} = A \mathbf{x}_u + B \mathbf{x}_v$  be a local
vector field on $\Sigma$. The divergence, $\nabla_{\Sigma} \cdot
\mathbf{X}$, of $\mathbf{X}$ is defined as a function
$\nabla_{\Sigma} \cdot \mathbf{X} : \Sigma \rightarrow \mathbb{R}$
given by the trace of the linear mapping $\mathbf{Y}(\mathbf{p})
\rightarrow \nabla_{\mathbf{Y}(\mathbf{p})} \mathbf{X}$ for
$\mathbf{p} \in \Sigma$. A direct computation gives

\begin{equation}\label{gradient_X}
\nabla_{\Sigma} \cdot X = \frac{1}{\sqrt{EG-F^2}}\left ( \frac{\partial}{\partial u} ( A\sqrt{EG-F^2})
 + \frac{\partial }{\partial v}(B\sqrt{EG-F^2}) \right ).
\end{equation}
The LB operator acting on the function $g$ is defined by
\begin{equation}
\Delta_{\Sigma} g = \nabla_{\Sigma} \cdot \nabla_{\Sigma}g.
\end{equation}
for all smooth function $g$ on $\Sigma$. A direct computation yields
the following local representation for the LB operator on a smooth
function $g$:
\begin{equation}\label{laplace_g}
\begin{array}{ll}
\Delta_{\Sigma} g = & \frac{1}{\sqrt{EG-F^2}} \left [ \frac{\partial}{\partial u} ( \frac{G}{\sqrt{EG-F^2}} \frac{\partial g}{\partial u})
- \frac{\partial}{\partial u} (\frac{F}{\sqrt{EG-F^2}} \frac{\partial g}{\partial v} ) \right ] \cr\cr
& + \frac{1}{\sqrt{EG-F^2}} \left [ \frac{\partial}{\partial v} ( \frac{E}{\sqrt{EG-F^2}} \frac{\partial g}{\partial v})
- \frac{\partial}{\partial v} (\frac{F}{\sqrt{EG-F^2}} \frac{\partial g}{\partial u} ) \right ]. \end{array}
\end{equation}

\section{An  $O(r)$-LTL configuration method}

In this section, we shall introduce a new algorithm to solve the
eigenpair problems, Equations (\ref{problem_LB_1}) and
(\ref{problem_LB_2}), by the LTL configuration method.

Consider a triangular surface mesh $S=(V,F)$, where
$V=\{\mathbf{v}_i| 1\leq i \leq n_V \}$ is the list of vertices and
$F=\{ T_k | 1 \leq k \leq n_F\}$ is the list of triangles.

To describe the local tangential lifting (LTL) method, we introduce
the approximating tangent plane $TS_A(\mathbf{v})$ and the local
tangential polygon $P_A(\mathbf{v})$ at the vertex $\mathbf{v}$  of
$S$ as follows:
\begin{enumerate}
\item The normal vector  $\mathbf{N}_A(\mathbf{v})$ at the vertex $\mathbf{v}$  in $S$  is given by
\begin{equation}\label{N_A}
\mathbf{N}_A(\mathbf{v}) = \frac{\sum_{T \in T(\mathbf{v})}\omega_T\mathbf{N}_T}{\|\sum_{T \in
T(\mathbf{v})}\omega_T\mathbf{N}_T\|}
\end{equation}
where  $T(\mathbf{v})$ is the set of triangles that contain the
vertex $\mathbf{v}$, $\mathbf{N}_T$ is the unit normal to a
triangle face $T$  with $\langle \mathbf{N}_T, \mathbf{N}_T'
\rangle > 0$ for all $T, T' \in T(\mathbf{v})$ and the centroid
weight is given in \cite{Chen,Chen2} by
\begin{equation}
\omega_T = \frac{\frac{1}{\|\mathbf{G}_T - \mathbf{v}\|^2}}{\sum_{\tilde{T} \in T(\mathbf{v})}
\frac{1}{\|\mathbf{G}_{\tilde{T}} - \mathbf{v} \|^2}}
\end{equation}
where $\mathbf{G}_T$ is the centroid of the triangle face $T$
determined by
\begin{equation}
\mathbf{G}_T = \frac{\mathbf{v}+\mathbf{v}_i+\mathbf{v}_j}{3}.
\end{equation}
Note that the letter $A$ in the notation $\mathbf{N}_A(v)$
stands for the word "Approximation".

\item The approximating tangent plane $TS_A(\mathbf{v})$  of $S$  at $\mathbf{v}$   is now determined by
$TS_A(\mathbf{v}) = \{ \mathbf{w} \in \mathbb{R}^3 | \mathbf{w}
\bot \mathbf{N}_A(\mathbf{v}) \}$.

\item The local tangential polygon $P_A(\mathbf{v})$  of $\mathbf{v}$ in $TS_A(\mathbf{v})$ is formed
by the vertices $\bar{\mathbf{v}}_i$ which is the lifting vertex
of $\mathbf{v}_i$ adjacent to $\mathbf{v}$  in $V$:
\begin{equation}
\bar{\mathbf{v}}_i = (\mathbf{v}_i - \mathbf{v}) - < \mathbf{v}_i - \mathbf{v},\mathbf{N}_A(\mathbf{v})>\mathbf{N}_A(\mathbf{v})
\end{equation}
as in figure \ref{figure_tangential_polygon}.

\begin{figure}[!t]
\centering
\includegraphics[width=2.5in]{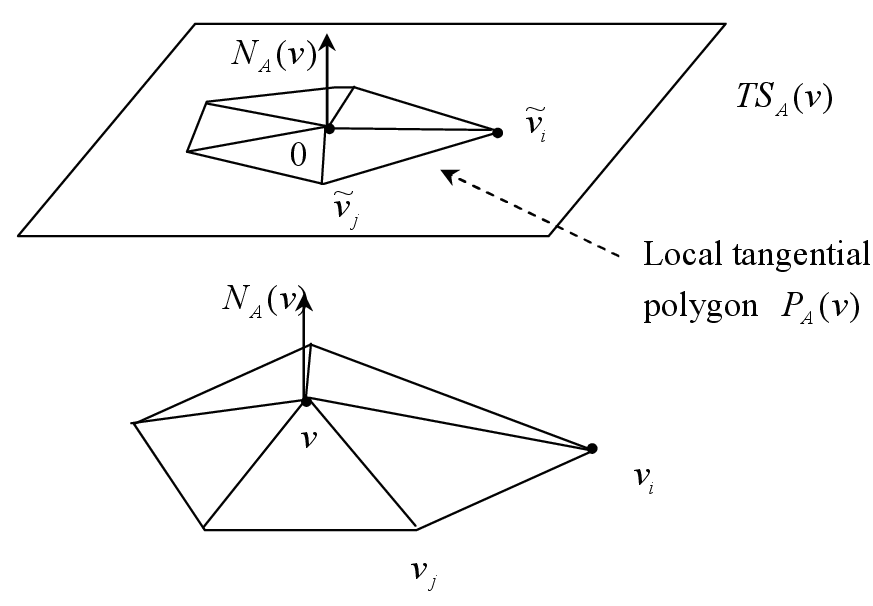}
\caption{The local tangential polygon $P_A(\mathbf{v})$}\label{figure_tangential_polygon}

\end{figure}

\item We can choose an orthonormal basis $\mathbf{e}_1,\mathbf{e}_2$ for the tangent plane
$TS_A(\mathbf{v})$ of $S$ at $\mathbf{v}$ and obtain an
orthonormal coordinates $(x,y)$ for vectors $\mathbf{w} \in
TS_A(\mathbf{v})$ by $\mathbf{w} = x\mathbf{e}_1+y\mathbf{e}_2$.
We set $\bar{\mathbf{v}}_i = x_i\mathbf{e}_1 + y_i\mathbf{e}_2$
with respect to the orthonormal basis $\mathbf{e}_1,
\mathbf{e}_2$.

\end{enumerate}

Now we explain how to lift locally a function defined on $V$  to the
local tangential polygon $P_A(\mathbf{v})$. Consider a function
$\phi$ on $V$. We will lift locally the function $\phi$ to a
function of two variables , denoted by $\bar{\phi}$, on the vertices
$\bar{\mathbf{v}}_i$ in $P_A(\mathbf{v})$ by simply setting
\begin{equation}
\bar{\phi}(x_i,y_i) = \phi(\mathbf{v}_i)
\end{equation}
and $\bar{\phi}(\vec{0}) = \phi(\mathbf{v})$  where  $\vec{0}$ is
the origin of $TS_A(\mathbf{v})$ . Then one can extend the function
$\bar{\phi}$ to be a piecewise linear function on the whole polygon
$P_A(\mathbf{v})$ in a natural and obvious way.

Next, we introduce the configuration matrices for $\Delta_{\Sigma}
\phi$ and $\nabla_{\Sigma} \cdot (h \nabla_{\Sigma} \phi)$ in
Equations (\ref{problem_LB_1}) and (\ref{problem_LB_2}).

\subsection{ The Configuration matrix for $\Delta_{\Sigma}\phi$ }\label{section_configuration_LB_1}
According to  the LTL method, the differential quantities at a point
on curved surfaces  correspond to  planar differential quantities in
$\mathbb{R}^2$. Hence, we only need to estimate the Laplace-Beltrami
operator on planar triangular meshes. Given a $C^3$ function $f$ on
an open domain $\Omega \subset \mathbb{R}^2$ with the origin $(0,0)
\in \Omega$, Taylor's expansion for two variables $x$ and $y$ gives

\begin{equation}\label{Taylor_expansion}
\begin{array}{rl}
f(x,y) = & f(0,0) + xf_x(0,0) + yf_y(0,0) \cr\cr
&  + \frac{x^2}{2}f_{xx}(0,0)+ xyf_{xy}(0,0) \cr\cr
& +\frac{y^2}{2} f_{yy}(0,0) + O(r^3) \end{array}
\end{equation}
when $r= x^2+y^2$ is small.

Consider a family of neighboring points $(x_j,y_j) \in \Omega$,
$j=1,2, \cdots, n$, of the origin $(0,0)$. Take some constants
$\alpha_j$, $j=1,2, \cdots, n$, with $\sum^n_{j=1} \alpha_j =
\epsilon \neq 0$. Then one has
\begin{equation}
\begin{array}{ll}
& \sum\limits_{j=1}^n \alpha_j( f(x_j,y_j) - f(0,0)) \cr\cr
= &(\sum\limits_{j=1}^n \alpha_j x_j)f_x(0,0)+ (\sum\limits_{j=1}^n
\alpha_j y_j)f_y(0,0) \cr\cr

& + \frac{1}{2}(\sum\limits_{j=1}^n \alpha_j x_j^2)f_{xx}(0,0) +
(\sum\limits_{j=1}^n \alpha_j x_j y_j) f_{xy}(0,0) \cr\cr

&+\frac{1}{2}(\sum\limits_{j=1}^n \alpha_j y_j^2) f_{yy}(0,0) +
O(r^3),
\end{array}
\end{equation}
where $r = \max\limits_{j \in \{ 1,2,\cdots,n\} } \{
\sqrt{x_j^2+y_j^2} \}$.  To estimate  the Laplacian, $f_{xx}(0,0) +
f_{yy}(0,0)$, at $(0,0)$,  we choose the constants $\alpha_j$,
$j=1,2, \cdots, n$ with $\sum^n_{j=1} \alpha_j = \epsilon$, so that
they satisfy the following equations:

$$
\begin{array}{l}
  \sum\limits_{j=1}^n \alpha_j x_j = 0, \cr\cr
 \sum\limits_{j=1}^n \alpha_j y_j = 0, \cr\cr
   \sum\limits_{j=1}^n \alpha_j x_j y_j = 0,
\end{array}
$$
 and
$$  \sum\limits_{j=1}^n \alpha_j x_j^2 = \sum\limits_{j=1}^n \alpha_j y_j^2$$
 or equivalently
$$
\sum\limits_{j=1}^n
\alpha_j(x_j^2- y_j^2) = 0
$$

One can rewrite these equations in a matrix form with the condition
$\sum_{j=1}^n \alpha_j = \epsilon$ and obtain the following
equation:

\begin{equation}\label{configuration_LB}
\begin{pmatrix}
x_1, & x_2, & \cdots, & x_n \cr y_1, & y_2, & \cdots, & y_n \cr
x_1y_1, & x_2y_2, & \cdots, & x_ny_n \cr

x_1^2 - y_1^2, & x_2^2 - y_2^2, & \cdots, & x_n^2-y_n^2 \cr 1, & 1,
& \cdots, & 1
\end{pmatrix}\begin{pmatrix} \alpha_1 \cr \alpha_2 \cr \vdots \cr
\alpha_n \end{pmatrix} = \begin{pmatrix} 0 \cr 0 \cr 0 \cr 0 \cr
\epsilon\end{pmatrix}.
\end{equation}
The solutions $\alpha_j$ of this equation allow us to obtain a
formula for the Laplacian $\Delta f(0,0)$:
\begin{equation}\label{laplace_1}
\begin{array}{ll}
\Delta f(0,0) & = f_{xx}(0,0) + f_{yy}(0,0) \cr\cr  & =
\frac{2\sum\limits_{j=1}^n
\alpha_j (f(x_j,y_j)-f(0,0))}{\sum\limits_{j=1}^n \alpha_j x_j^2}
+O(r)
\end{array}
\end{equation}

\begin{Rmk}
\begin{enumerate}
\item Equation (\ref{configuration_LB}) is called the configuration equation  of the Laplace-Beltrami operator.
We call the matrix $\begin{pmatrix} x_i, & y_i, & x_i y_i, &
x_i^2 - y_i^2 \end{pmatrix}^T$ in Equation
(\ref{configuration_LB})  the configuration matrix of
Laplace-Beltrami operator and the solution $\begin{pmatrix}
\alpha_i \end{pmatrix}$ in Equation (\ref{configuration_LB}) the
configuration coefficients of the Laplace-Beltrami operator,
respectively.

\item
For the reason of symmetry, Equation (\ref{laplace_1}) also gives
\begin{equation}\label{laplace_2}
\Delta f(0,0)  = \frac{4\sum\limits_{j=1}^n
\alpha_j (f(x_j,y_j)-f(0,0))}{\sum\limits_{j=1}^n \alpha_j(
x_j^2+y_j^2) } +O(r)
\end{equation}
since we have $\sum\limits_{j=1}^n \alpha_j x_j^2 =
\sum\limits_{j=1}^n \alpha_j y_j^2$ .

\item  For simplicity, the scalar $\epsilon$ in Equation (\ref{configuration_LB}) can be chosen to be $1$.

\item It is worth to point out that Equation (\ref{laplace_2}) is an
generalization of the well-known 5-point Laplacian formula. In
the 5-point Laplacian case, we have the origin $(0,0)$ along
with 4 neighboring points $(s,0)$, ($0,s)$ $(-s,0)$ and $(0,-s)$
for sufficiently small positive number $s$. One can find a
solution $\alpha_j = \frac{1}{4}$ for $j=1,2,3,4$, in this case.
\end{enumerate}
\end{Rmk}

Now,  let $\Sigma$ be a regular surface with a  triangular surface
mesh $S=(V,F)$ of $\Sigma$. The set $V = \{\mathbf{v}_i
|i=1,2,\cdots n_V\}$ is the list of vertices on $S$, $F =\{T_k |
k=1,2,\cdots,n_F\}$ is the list of triangles on $S$ and $N(i)$ is
the number of 1-neighbors of $\mathbf{v}_i$ on $S$.  Suppose that
$\phi$ is a $C^3$ function  on $S$. For each vertex $\mathbf{v}_i$
on $S$, $\{\bar{\mathbf{v}}_{i,j}\}_{j=1}^{N(i)}$ is the tangential
polygon of the neighbors of $\mathbf{v}_i$ with coordinates
$\{(x_{i,j},y_{i,j}) | i=1,2,\cdots N(i) \}$ and
$\phi(x_{i,j},y_{i,j}) = \phi(\mathbf{v}_{i,j})$ is the lifting
function of $\phi$. By the configuration equation of the Laplacian
in Equation (\ref{configuration_LB}), the Laplace-Beltrami operator,
$\Delta_S \phi$ on $S$, is defined by
\begin{equation}\label{laplace_phi}
\Delta_S \phi(\mathbf{v}_i) = \frac{2\sum_{j=1}^{N(i)} \alpha_{i,j}(\phi(x_{i,j}, y_{i,j}) - \phi(0,0) )}{\sum_{j=1}^{N(i)}(\alpha_{i,j}x_{i,j}^2)}.
\end{equation}
Then one can prove
\begin{Thm}\label{convergence_LB}
Given a smooth function $\phi$ on a closed  regular surface $\Sigma$
and a triangular surface mesh $S=(V,F)$  with mesh size $r$, one has
\begin{equation}
\Delta_{\Sigma} \phi(\mathbf{v}) = \Delta_S \phi(\mathbf{v}) +O(r)
\end{equation}
where the discrete LB operator $\Delta_S \phi(\mathbf{v})$ is given
in Equation (\ref{laplace_phi}).
\end{Thm}

We will prove Theorem \ref{convergence_LB} by the following $4$
Lemmas. Indeed, Theorems \ref{thm3} and \ref{thm4} in this section
can also be proved in a similar way.

Given a smooth function $h$ on a regular surface $\Sigma$, we can
lift $h$ via the exponential map $\exp_\mathbf{p}$ locally to obtain
a smooth function $\hat{h}$ defined on $W \in T\Sigma(\mathbf{p})$
by setting
\begin{equation}
\hat{h}(\mathbf{w}) = h(\exp_{\mathbf{p}}(\mathbf{w}))
\end{equation}
for $\mathbf{w} \in W$. Fix an orthonormal basis
$\tilde{\mathbf{e}}_1, \tilde{\mathbf{e}}_2$
 for the tangent space  $T\Sigma(\mathbf{p})$. This gives us a coordinate system on  $T\Sigma(\mathbf{p})$.
Namely, for $\mathbf{w}\in W$ we have
$\mathbf{w}=x\tilde{\mathbf{e}}_1+y\tilde{\mathbf{e}}_2$ for two
constants $x$ and $y$. Without loss of ambiguity, we can identify
the vector $\mathbf{w} \in W$ with the vector $(x,y)$ with respect
to the orthonormal basis $\tilde{\mathbf{e}}_1,
\tilde{\mathbf{e}}_2$. In this way, the function $\hat{h}$ can also
give us a smooth function $\tilde{h}$ of two variables $x$ and $y$
by defining
\begin{equation}\label{tilde_h}
\tilde{h}(x,y) = \hat{h}(\mathbf{w})
\end{equation}
for $\mathbf{w} = x\tilde{\mathbf{e}}_1+y\tilde{\mathbf{e}}_2$.
Using these notations, we will prove

\begin{Lem}\label{Lemma1}
One has
 \begin{equation} \Delta_{\Sigma} h(\mathbf{p}) = \Delta \hat{h}(0) =
\Delta \tilde{h}(0,0).
\end{equation}
\end{Lem}

\begin{proof}
It is well-known that the LB operator $\Delta_{\Sigma}
h(\mathbf{p})$ acting on a smooth  function $h$ at a point
$\mathbf{p}$ can be computed from the second derivatives of $h$
along any two perpendicular geodesics with unit speed. See do  Carmo
\cite{Docarmo2} for details. Indeed, we consider the following two
perpendicular geodesics with unit speed in $\Sigma$ by using the
orthonormal vectors $\tilde{\mathbf{e}}_1, \tilde{\mathbf{e}}_2$:
\begin{equation}
c_i(t) = \exp_{\mathbf{p}} (t\tilde{\mathbf{e}}_i), ~~ i=1,2
\end{equation}
with $c_i(0) = \mathbf{p}$  and $\frac{dc_i}{dt}(0) =
\tilde{\mathbf{e}}_i$. One has
\begin{equation}
\begin{split}
\Delta_{\Sigma} h(\mathbf{p}) = & \frac{d^2}{dt^2}h(c_1(t))|_{t=0} +
\frac{d^2}{dt^2}h(c_2(t))|_{t=0} \\
= & \frac{d^2}{dt^2}\hat{h}(t\tilde{\mathbf{e}}_1)|_{t=0} +
\frac{d^2}{dt^2}\hat{h}(t\tilde{\mathbf{e}}_2)|_{t=0} \\
= & \Delta \hat{h}(\vec{0}) \\
= & \frac{\partial^2 \tilde{h}}{\partial x^2}(0,0) +
\frac{\partial^2 \tilde{h}}{\partial y^2}(0,0) \\
= & \Delta \tilde{h}(0,0).
\end{split}
\end{equation}
\end{proof}

Next we consider a triangular surface mesh $S=\{V,F\}$ for the
regular surface $\Sigma$, where $V=\{v_i|1 \leq i \leq n_V\}$  is
the list of vertices and $F = \{T_k \ 1 \leq k \leq n_F \}$ is the
list of triangles and the mesh size is less than $r$. Fix a vertex
$\mathbf{v}$ in $V$. For each face $T \in F$ containing
$\mathbf{v}$, we have
\begin{equation}
\mathbf{N}_{\Sigma}(\mathbf{v}) = \mathbf{N}_T + O(r)
\end{equation}
where $\mathbf{N}_{\Sigma} (\mathbf{v})$ is the unit normal vector
of the true tangent plane $T\Sigma (\mathbf{v})$ of  $\Sigma$ at
$\mathbf{v}$ and $\mathbf{N}_T$ is the unit normal vector of the
face $T$. Since the approximating normal vector
$\mathbf{N}_A(\mathbf{v})$, defined in section 3 is a weighted sum
of these neighboring face normals $\mathbf{N}_T$, we have
\begin{Lem}\label{Lemma2}
One has
\begin{equation}\label{equ_lemma2}
\mathbf{N}_{\Sigma} (\mathbf{v}) = \mathbf{N}_A{\mathbf{v}} + O(r).
\end{equation}
\end{Lem}
Due to this lemma, the orthonormal basis $\tilde{\mathbf{e}}_1,
\tilde{\mathbf{e}}_2$ for the tangent plane $T\Sigma(\mathbf{v})$
will give us an orthonormal basis $\mathbf{e}_1,\mathbf{e}_2$ for
the approximating tangent space $T\Sigma_A(\mathbf{v})=\{ \mathbf{w}
\in \mathbb{R}^3 | \mathbf{w} \bot \mathbf{N}_A(\mathbf{v}) \}$ by
the Gram-Schmidt process in linear algebra:
$$ \mathbf{e}_1 = \frac{\tilde{\mathbf{e}}_1 -<\tilde{\mathbf{e}}_1, \mathbf{N}_A(\mathbf{v})>\mathbf{N}_A(\mathbf{v})}{\|\tilde{\mathbf{e}}_1 -<\tilde{\mathbf{e}}_1,
\mathbf{N}_A(\mathbf{v})>\mathbf{N}_A(\mathbf{v})\|},$$
and
$$\mathbf{e}_2= \frac{\tilde{\mathbf{e}}_2 -<\tilde{\mathbf{e}}_2, \mathbf{N}_A(\mathbf{v})>\mathbf{N}_A(\mathbf{v})
-
<\tilde{\mathbf{e}_2},\mathbf{e}_1>\mathbf{e}_1}{\|\tilde{\mathbf{e}}_2
-<\tilde{\mathbf{e}}_2,
\mathbf{N}_A(\mathbf{v})>\mathbf{N}_A(\mathbf{v}) -
<\tilde{\mathbf{e}_2},\mathbf{e}_1>\mathbf{e}_1\|}.$$ Logically
speaking, one can first choose an orthonormal basis
$\mathbf{e}_1,\mathbf{e}_2$ for the approximating tangent space
$TS_A(v)$ and then apply the Gram-Schmidt process to obtain an
orthonormal basis $\tilde{\mathbf{e}}_1, \tilde{\mathbf{e}}_2$ for
the tangent plane $T\Sigma(\mathbf{v})$. In either way, we always
have by Lemma \ref{Lemma2} the following relations.

\begin{Lem}\label{Lemma3}
One has
\begin{equation}\label{equ_lemma3}
\tilde{\mathbf{e}}_i = \mathbf{e}_i + O(r), ~~ i=1,2.
\end{equation}
\end{Lem}
Consider a neighboring vertex $\mathbf{v}_i$  of $\mathbf{v}$  in
$V$. For $r$ small enough, we can use the inverse of the exponential
map $\exp_{\mathbf{p}}$ to lift the vertex $\mathbf{v}_i$ up to the
tangent plane $T\Sigma(\mathbf{v})$ and obtain

$$\tilde{\mathbf{v}}_i = \exp^{-1}_{\mathbf{v}} (\mathbf{v}_i) \in T\Sigma(\mathbf{v})$$
 and
$$ \tilde{\mathbf{v}}_i = \tilde{x}_i \tilde{\mathbf{e}}_1 + \tilde{y}_i \tilde{\mathbf{e}}_2$$
for some constants. As discussed in section 3, we can also lift the
vertex $\mathbf{v}_i$ up to the approximating tangent plane
$T\Sigma_A(\mathbf{v})$ and get

$$ \bar{\mathbf{v}}_i = (\mathbf{v}_i - \mathbf{v}) - <\mathbf{v}_i - \mathbf{v}, \mathbf{N}_A(\mathbf{v})>\mathbf{N}_A(\mathbf{v})$$
and for some constants $x_i, y_i$. Then Lemmas \ref{Lemma2} and
\ref{Lemma3} yield
\begin{Lem}\label{Lemma4}
One has
\begin{equation}
\left \{ \begin{split}
\tilde{x}_i & = x_i + O(r^2) \\
\tilde{y}_i & = y_i + O(r^2).
\end{split}\right.
\end{equation}
\end{Lem}
Using these relations, one can solve the configuration equation
(\ref{configuration_LB}) for $(\tilde{x}_i,\tilde{y}_i)$ and
$(x_i,y_i)$ respectively and obtain their corresponding solutions
$\tilde{\alpha}_i$ and $\alpha_i$ with the relation
\begin{equation}
\tilde{\alpha}_i = \alpha_i + O(r).
\end{equation}
Note that the lifting function $\tilde{h}$ is a smooth function of
two variables $x$ and $y$. Equation (\ref{laplace_2}) now gives an
approximation of the Laplacian $\Delta \tilde{h}(0,0)$:
\begin{equation}
\Delta \tilde{h}(0,0) = \frac{4\sum\limits_{i=1}^n \tilde{\alpha}_i(\tilde{h}(x_i,y_i) -
\tilde{h}(0,0))}{\sum\limits_{i=1}^n\tilde{\alpha}_i( \tilde{x_i}^2+\tilde{y_i}^2)} + O(r).
\end{equation}
The relations (\ref{tilde_h}), (\ref{equ_lemma2}) and
(\ref{equ_lemma3}) imply
\begin{equation}
\Delta \tilde{h}(0,0) = \frac{4\sum\limits_{i=1}^n \alpha_i(h(\mathbf{v}_i) -
h(\mathbf{v}))}{\sum\limits_{i=1}^n\alpha_i( x_i^2+y_i^2)} + O(r).
\end{equation}
This along with Lemma \ref{Lemma1} proves Theorem
\ref{convergence_LB}.

For each vertex $\mathbf{v}_i \in V$, we have
\begin{equation}\label{approximate_LB_phi}
\Delta_{\Sigma} \phi(\mathbf{v}_i) = \frac{2\sum_{j=1}^{N(i)} \left (  \alpha_{i,j}(\phi(\mathbf{v}_{i,j})
- \phi(\mathbf{v}_i) )\right ) }{\sum_{j=1}^{N(i)}\left ( \alpha_{i,j} x_{i,j}^2  \right )} + O(r).
\end{equation}
Denote $\omega_i = \frac{2}{\sum_{j=1}^{N(i)} \left ( \alpha_{i,j}
x_{i,j}^2 \right ) }$. Since $\sum_{j=1}^{N(i)} \alpha_{i,j} = 1$,
Equation (\ref{approximate_LB_phi}) can be rewritten as
\begin{equation}
\Delta_{\Sigma} \phi(\mathbf{v}_i) = \omega_i \left [  \left ( \sum_{j=1}^{N(i)} (\alpha_{i,j} \phi(\mathbf{v}_{i,j}) \right ) - \phi(\mathbf{v}_i) \right ] +O(r).
\end{equation}
Furthermore  the vector $ ( \alpha_{i,1}, \alpha_{i,2}, \cdots,
\alpha_{i,N(i)} )$ can be easily extended to a $1\times n_V$ vector
$( a_{i,1}, a_{i,2}, \cdots , a_{i,n_V})$ by
\begin{equation}\label{a_ij}
 a_{i,k} = \left \{ \begin{array}{ll} \alpha_{i,j} & \mbox{ if there exists } j \in \{1,2,\cdots, N(i)\} \mbox{ such that } v_k = v_{i,j}, \cr
 -1 & \mbox{ if } i=k, \cr
 0 & \mbox{ otherwise.} \end{array} \right.
\end{equation}
Obviously,
\begin{equation}
 \Delta_{\Sigma} \phi(\mathbf{v}_i) =  \omega_i  \left [ (a_{i,1}, a_{i,2}, \cdots, a_{i,n_V} ) (\phi(\mathbf{v}_1),
  \phi(\mathbf{v}_2), \cdots, \phi(\mathbf{v}_{n_V}) )^T \right ] + O(r).
\end{equation}

\begin{Rmk}
$\{1, 2, \cdots, n_V \}$ is the set of indices of all vertices in
$V$ and $\mathbf{v}_i$ denotes the ith vertex in $V$. For each $i
\in \{ 1, 2, \cdots, n_V\}$, $\{1, 2, \cdots, N(i)\}$ is the set of
indices of one-neighbors of $\mathbf{v}_i$ and $\mathbf{v}_{i,j}$
denotes the jth one-neighbor of $\mathbf{v}_i$ in $\{
\mathbf{v}_{i,1}, \mathbf{v}_{i,2}, \cdots, \mathbf{v}_{i,N(i)} \}$.
Obviously, every one-neighbor $\mathbf{v}_{i,j}$ of $\mathbf{v}_i$
is corresponding to a unique vertex $\mathbf{v}_k$ in $V$ while $j$
and $k$ may be not equal.
\end{Rmk}

This implies that
\begin{equation}\label{eign_phi_1}
\left (
\begin{array}{c} \Delta_\Sigma \phi(\mathbf{v}_1) \cr \Delta_\Sigma \phi(\mathbf{v}_2) \cr \vdots \cr \Delta_\Sigma \phi(\mathbf{v}_{n_V}) \end{array} \right ) =
W  A \left ( \begin{array}{c} \phi(\mathbf{v}_1) \cr \phi(\mathbf{v}_2) \cr \vdots \cr \phi(\mathbf{v}_{n_V}) \end{array}\right ) + O(r),
\end{equation}
where $W = diag (\omega_1, \cdots, \omega_{n_V})$ and $A=(a_{i,k})$
are two $n_V \times n_V$ matrices. For simplicity, we rewrite
Equation (\ref{eign_phi_1}) as
\begin{equation}
\Delta_{\Sigma} \phi (V) = (WA) \phi(V) + O(r).
\end{equation}
Hence, we have an eigenvalue approximation result by the method
discussed in \cite{Strange1}.

\begin{Thm}\label{thm3}
Let $\Sigma$ be a closed regular surface, $S=(V,F)$ be a triangular
mesh of $\Sigma$ with mesh width $r$. If $\lambda_i$ is the ith
eigenvalue of the Laplace-Beltrami operator on $\Sigma$ and
$\bar{\lambda}_i$ is the ith eigenvalue of the matrix $WA$., then we
have, for sufficiently small $r>0$,
$$ \lambda_i = \bar{\lambda}_i + O(r).$$
\end{Thm}

\begin{Rmk}
The matrix equation

\begin{equation}\label{configuration_matrix_LB}
\left (
\begin{array}{c} \Delta_\Sigma \phi(\mathbf{v}_1) \cr \Delta_\Sigma \phi(\mathbf{v}_2) \cr \vdots \cr \Delta_\Sigma \phi(\mathbf{v}_{n_V}) \end{array} \right ) =
W  A \left ( \begin{array}{c} \phi(\mathbf{v}_1) \cr \phi(\mathbf{v}_2) \cr \vdots \cr \phi(\mathbf{v}_{n_V}) \end{array}\right )
\end{equation}
 is called the configuration equation of the
Laplace-Beltrami operator at $\mathbf{v}_i$ on $S$. The constants
$\alpha_{i,j}$ are called the configuration coefficients of
Laplace-Beltrami operator at $\mathbf{v}_i$ on $S$. The $n_V \times
n_V$ matrix $A$ defined in Equation (\ref{a_ij}) is called the
configuration matrix of the Laplace-Beltrami operator on $S$.
\end{Rmk}

\subsection{ The Configuration matrix for $\nabla_{\Sigma}\cdot (h \nabla_{\Sigma}\phi)$ }\label{section_configuration_LB_2}
Let $h$ be a bounded smooth function defined on a regular surface
$\Sigma$. We introduce the configuration matrix of the quantity
$\nabla_{\Sigma}\cdot (h \nabla_{\Sigma}\phi)$ by a similar method
as in subsection \ref{section_configuration_LB_1}. First, let us
consider two smooth functions $f$ and $g$ defined on an open domain
$\Omega$ in $\mathbb{R}^2$ with the original point $(0,0) \in
\Omega$. Since $\nabla_{\Sigma}\cdot (h \nabla_{\Sigma}\phi) =
\nabla_{\Sigma} h \cdot \nabla_{\Sigma} \phi + h \Delta_{\Sigma}
\phi$, we need to estimate the quantity $ \nabla g \cdot \nabla f +
g \Delta f = g_xf_x+g_yf_y + g(f_{xx}+f_{yy})$ in $\mathbb{R}^2$.
Taylor expansions of $f$ and $g$ are given by

\begin{equation}\label{Taylor_expansion_f2}
\begin{array}{rl}
f(x,y) - f(0,0) = &   xf_x(0,0) + yf_y(0,0) \cr\cr
&  + \frac{x^2}{2}f_{xx}(0,0)+ xyf_{xy}(0,0) \cr\cr
& +\frac{y^2}{2} f_{yy}(0,0) + O(r^3) \end{array}
\end{equation}
and
\begin{equation}\label{Taylor_expansion_fg}
\left \{
\begin{array}{rl}
f(x,y) - f(0,0) = xf_x(0,0) + y f_y(0,0) +O(r^2) \cr\cr
g(x,y) - g(0,0) = xg_x(0,0) + y g_y(0,0) + O(r^2),
\end{array} \right.
\end{equation}
when $r= x^2+y^2$ is small. From Equation
(\ref{Taylor_expansion_fg}), one has
\begin{equation}\label{Taylor_expansion_fg2}
\begin{array}{ l}
   \left ( f(x,y)-f(0,0) \right ) \left ( g(x,y)-g(0,0) \right ) \cr\cr
 =   \left ( xf_x(0,0) + y f_y(0,0) \right ) \left (  xg_x(0,0) + y g_y(0,0) \right ) +O(r^3) \cr\cr
 =   x^2f_x(0,0)g_x(0,0) + xy\left ( f_x(0,0)g_y(0,0) + f_y(0,0)g_x(0,0) \right )\cr\cr
   ~~ + y^2f_y(0,0)g_y(0,0)+O(r^3).
 \end{array}
\end{equation}
These imply
\begin{equation}
\begin{array}{ l}
 \left ( f(x,y)-f(0,0) \right ) \left ( g(x,y)-g(0,0) \right ) + 2g(0,0)(f(x,y)-f(0,0)) \cr\cr
 =   2xg(0,0)f_x(0,0) +2yg(0,0)f_y(0,0) + xy (2g(0,0)f_{xy}(0,0)+ f_x(0,0)g_y(0,0)  \cr\cr
   ~~ + f_y(0,0)g_x(0,0) ) + x^2(g(0,0)f_{xx}(0,0)+  f_x(0,0)g_x(0,0) ) \cr\cr
   ~~   + y^2(g(0,0)f_{yy}(0,0) +  f_y(0,0)g_y(0,0) ) +O(r^3).
  \end{array}
\end{equation}
Consider a family of neighboring points $(x_j, y_j) \in \Omega$,¸
Ħ$j = 1, 2, \cdots, n$, of the origin $(0, 0)$. Take some constants
ƒ¿$\beta_j, j = 1, 2, \cdots, n$ with $\sum^n_{j=1} \beta_j =1$. We
have
\begin{equation}
\begin{array}{ l}
 \sum\limits_{j=1}^n \left [  \left ( f(x_i,y_i)-f(0,0) \right ) \left ( g(x_i,y_i)-g(0,0) \right ) + 2g(0,0)(f(x_i,y_i)-f(0,0)) \right ] \cr\cr
 =  (\sum\limits_{j=1}^n \beta_j x_j ) 2g(0,0)f_x(0,0) + (\sum\limits_{j=1}^n \beta_j y_j) 2g(0,0)f_y(0,0) \cr\cr
  ~~  + (\sum\limits_{j=1}^n \beta_j x_j y_j) (2g(0,0)f_{xy}(0,0)+ f_x(0,0)g_y(0,0)  + f_y(0,0)g_x(0,0) ) \cr\cr
 ~~  + ( \sum\limits_{j=1}^n \beta_j x_j^2)(g(0,0)f_{xx}(0,0)+   f_x(0,0)g_x(0,0) ) \cr\cr
 ~~  + (\sum\limits_{j=1}^n \beta_j y_j^2) (g(0,0)f_{yy}(0,0) + f_y(0,0)g_y(0,0) ) +O(r^3).
  \end{array}
\end{equation}
To compute $\nabla g \cdot \nabla f + g \Delta f$ at $(0,0)$, we
choose the constants $\beta_j$, $j = 1, 2, \cdots, n$, so that they
satisfy the following equations:

\begin{equation}\label{configuration_LB2}
\begin{pmatrix}
x_1, & x_2, & \cdots, & x_n \cr y_1, & y_2, & \cdots, & y_n \cr
x_1y_1, & x_2y_2, & \cdots, & x_ny_n \cr

x_1^2 - y_1^2, & x_2^2 - y_2^2, & \cdots, & x_n^2-y_n^2 \cr 1, & 1,
& \cdots, & 1
\end{pmatrix}\begin{pmatrix} \beta_1 \cr \beta_2 \cr \vdots \cr
\beta_n \end{pmatrix} = \begin{pmatrix} 0 \cr 0 \cr 0 \cr 0 \cr
1\end{pmatrix},
\end{equation}
The solutions $\beta_j$ of this equation gives a formula for $\nabla
g \cdot \nabla f + g \Delta f$ at $(0,0)$:
\begin{equation}
\begin{array}{ll}
& \left ( \nabla g \cdot \nabla f + g \Delta f \right )(0,0) \cr\cr
= & g_x(0,0)f_x(0,0)+ g_y(0,0)f_y(0,0) +g(0,0)\left [ f_{xx}(0,0)+f_{yy}(0,0) \right ] \cr\cr
= & \frac{1}{\sum\limits_{j=1}^n \beta_j x_j^2 } \sum\limits_{j=1}^n  [  \left ( f(x_j,y_j)-f(0,0) \right ) \left ( g(x_j,y_j)-g(0,0) \right ) \cr\cr
 & + 2g(0,0)(f(x_j,y_j)-f(0,0))  ]   +O(r)   \cr\cr
 = &  \frac{\sum\limits_{j=1}^n    \left ( f(x_j, y_j)-f(0,0) \right ) \left ( g(x_j, y_j) + g(0,0) \right )   }{\sum\limits_{j=1}^n \beta_j x_j^2 } +O(r)
\end{array}
\end{equation}

Using the notations of  subsection \ref{section_configuration_LB_1},
the quantity $\Delta_S (h \Delta_S \phi)$ at $\mathbf{v}_i$ on $S$
is given by
\begin{equation}\label{gen_laplace_phi}
\nabla_S (h \nabla_S \phi) (\mathbf{v}_i) =
\frac{\sum\limits_{j=1}^{N(i)} \left (\beta_{i,j}( f(x_i,y_i)-f(0,0)
)   ( g(x_i,y_i) + g(0,0)   )   \right ) }{\sum\limits_{j=1}^{N(i)}
\beta_{i,j} x_{i,j}^2}. \end{equation}
Then one can prove
\begin{Thm}\label{thm4}
Given two smooth functions $h, \phi$ on a closed regular surface
$\Sigma$ with a triangular surface mesh $S=(V,F)$, one has
\begin{equation}
\nabla_{\Sigma} (h \nabla_{\Sigma} )(\mathbf{v}_i) = \nabla_S (h \nabla_S \phi) (\mathbf{v}_i)  +O(r)
\end{equation}
where the quantity $\nabla_S (h \nabla_S \phi) (\mathbf{v}_i) $ is
given in Equation (\ref{gen_laplace_phi}) and $r$ is the mesh size
of $S$.
\end{Thm}

\begin{Rmk}
The error terms $O(r)$ in Theorems \ref{convergence_LB}, \ref{thm3}
and \ref{thm4} depend only on some geometric invariants of $S$ and
the function $f$ since $\Sigma$ is a closed regular surface.
\end{Rmk}

Similarly, We extend these scalars $\beta_{i,j}$ for each $i \in
\{1,2,\cdots, n_V \}$ and for each $j \in \{ 1, 2, \cdots, N(i)\}$
to a $n_V \times n_V $ matrix $B=(b_{i,k})$ with
\begin{equation}\label{gen_a_ij}
 b_{i,k} = \left \{ \begin{array}{ll} \beta_{i,j} & \mbox{ if there exists } j \in \{1, 2, \cdots, N(i)\} \mbox{ such that } v_k = v_{i,j}, \cr
 -1 & \mbox{ if } i=k, \cr
 0 & \mbox{ otherwise.} \end{array} \right. .
\end{equation}
And, we have
\begin{equation}
\left ( \begin{array}{c}
\nabla_{\Sigma} (h \nabla_{\Sigma} \phi)(\mathbf{v}_1) \cr
\nabla_{\Sigma} (h \nabla_{\Sigma} \phi)(\mathbf{v}_2) \cr
\vdots \cr
\nabla_{\Sigma} (h \nabla_{\Sigma} \phi)(\mathbf{v}_{n_V}) \cr
\end{array} \right ) =
WB \left ( \begin{array}{c} \phi(\mathbf{v}_1) \cr \phi (\mathbf{v}_2) \cr \vdots \cr \phi (\mathbf{v}_{n_V}) \end{array}\right )
+ O(r),
\end{equation}
where $W = diag (\omega_1, \cdots \omega_{n_V})$, $\omega_i
=\frac{1}{ \sum_{j=1}^{N(i)} \alpha_{i,j} x_{i,j}^2 }$.

\section{High-order approximations}

In this section we shall discuss how to use the LTL method to obtain
high-order approximations of these differential operators. The ideas
are very simple. First, we shall propose an algorithm to construct a
high-order approximation of the underlying surface . Second, we also
give a method to obtain a high-order approximation of smooth
functions on .Third, using these approximations, we can compute the
differential quantities under consideration with high-order
accuracies.

As before, we consider a triangular surface mesh $S=(V,F)$, of the
smooth surface $\Sigma$ where $V= \{\mathbf{v}_i | 1 \leq i \leq
n_V\}$ with mesh size $r > 0$ is the list of vertices and $F = \{
T_k | 1 \leq k \leq n_F \}$ is the list of triangles. To obtain a
high-order approximation of the underlying surface $\Sigma$ around a
vertex $\mathbf{v}$, we will try to construct a local
parametrization by representing the smooth surface $\Sigma$ as
locally a graph surface around the vertex $\mathbf{v}$. Let
$N_A(\mathbf{v})$ be the approximating normal vector at the vertex
$\mathbf{v}$ in $S$ as $\mathbf{v}$ in (\ref{N_A}). The
approximating tangent plane $TS_A(\mathbf{v})$ of $S$ at
$\mathbf{v}$ is given by $TS_A(\mathbf{v}) = \{ \mathbf{w} \in
\mathbb{R}^3 | \mathbf{w} \bot N_A(\mathbf{v})\}$.

We can choose an orthonormal basis $\mathbf{e}_1, \mathbf{e}_2$ for
the approximating tangent plane  $TS_A(\mathbf{v})$  of $S$  at
$\mathbf{v}$ and obtain an orthonormal coordinates $(x,y)$ for
vectors $\mathbf{w} \in TS_A(\mathbf{v})$ by $\mathbf{w} =
x\mathbf{e}_1 + y \mathbf{e}_2$. The approximating tangent plane is
nearly tangential to the surface and the $z$-coordinate is
orthogonal to the $xy$-plane, the approximating tangent plane
$TS_A(\mathbf{v})$, and corresponds to the height function $h$. That
is, for every point $p$ in  $\Sigma$ around $\mathbf{v}$ we can
assign it an $xyz$-coordinates as follows:

\begin{equation} (\mathbf{p}- \mathbf{v}) - \langle \mathbf{p}-\mathbf{v},
N_A(\mathbf{v}) \rangle N_A(\mathbf{v}) = x(\mathbf{p})\mathbf{e}_1
+ y(\mathbf{p})\mathbf{e}_2
\end{equation}
and

\begin{equation}
h (x(\mathbf{p}), y(\mathbf{p})) = \langle \mathbf{p} -
\mathbf{v}, N_A(\mathbf{v}) \rangle
\end{equation}

The $n$-ring ($n>1$) neighboring vertex $\mathbf{v}_i$  of
$\mathbf{v}$ in $V$ is now given as

\begin{equation}
\begin{array}{ccc}
x_i = x(\mathbf{v}_i), & y_i = y(\mathbf{v}_i), & z_i =h(x_i,y_i)
\end{array}
\end{equation}
in the $xyz$-space. To give locally a high-order surface
reconstruction of $\Sigma$, we only need to find a suitable
polynomial fitting for the height function $h$ with high-order
accuracy by using the local data  $z_i =h(x_i,y_i)$. This can be
done by employing the high-order Taylor expansion again as we did in
the previous section.

The height function $h$ can be approximating to  $k$th-order
accuracy about the origin $\mathbf{o} =(0,0)$ as
\begin{equation}
h(x,y) = \sum_{d=0}^k \sum_{m=0}^d c_{d,m} \frac{x^{d-m}y^m}{(d-m)!m!} + O(r^{d+1})
\end{equation}
where the constant $c_{d,m} = \frac{\partial^d}{\partial x^{d-m}
\partial y^m} h(0,0)$. In particular, one has $c_{0,0} = h(0,0)$. Since the regular surface $\Sigma$ is
smooth, the height function $h$ is locally smooth and has $k+1$
continuous derivatives. We assume that the $\frac{j}{2}$-ring ($j
\geq 2$) neighboring vertices $\mathbf{v}_i$  of $\mathbf{v}$ in $V$
is near the vertex $\mathbf{v}$. Since these vertices $\mathbf{v}_i$
sample the surface $\Sigma$ near the vertex $\mathbf{v}$, their
coordinates in the $xyz$-space allow us to obtain the equation:

\begin{equation}\label{high_h_i}
h(x_i,y_i) - h(0,0) =\sum_{d=0}^k \sum_{m=0}^d e_{d,m} \frac{x_i^{d-m}y_i^m}{(d-m)!m!}
\end{equation}
Let $n=\frac{(k+1)(k+2)}{2}-1$  be the total number of coefficients
$e_{d,m}$. We can choose $n$ nearest neighboring vertices
$\mathbf{v}_i$ of $\mathbf{v}$ to solve the equation
(\ref{high_h_i}) by the configuration method in section 3 and obtain
a set of solution for the coefficients $e_{d,m}$. That is, we can
find constants $\alpha_{s,t}^{i}$ $i=1,2, \cdots, n$ with $(s,t)
\neq (0,0)$ so that

\begin{equation}
\sum_{i=1}^n \alpha_{s,t}^{i} \frac{x_i^{d-m}y_i^m}{(d-m)!m!}
= \left \{ \begin{array}{ll}
1 & \mbox{ when } (d,m) = (s,t) \cr
0 & \mbox{ when } (d,m) \neq (s,t) \end{array}\right.
\end{equation}
Hence we have
\begin{equation}
e_{s,t} = \sum_{i=1}^n \alpha_{s,t}^{i} [ h(x_i,y_i) - h(0,0)]
\end{equation}

Moreover, we have the following approximation:
\begin{equation}
c_{d,m} = e_{d,m} +O(r^{k+1-d}).
\end{equation}
Next we discuss how to approximate a smooth function $\phi$  with
high-order accuracy. Consider a smooth function $\phi$ on $\Sigma$.
We can view $\phi$ as a function of $x$ and $y$ near the vertex
$\mathbf{v}$. That is, we have locally, for a point $\mathbf{p}$
near $\mathbf{v}$ with local coordinates $(x,y,h(x,y))$,

\begin{equation}
\phi(x,y) =\phi(\mathbf{p})
\end{equation}
Again we can use the local data $\phi (x_i,y_i) = \phi (\mathbf{p})$
to approximate the function $\phi$ with  $k$th-order accuracy about
the origin $\mathbf{o} =(0,0)$  by applying the configuration method
to the $k$th-order Taylor expansion of the function $\phi$  as we
just discussed above for the height function $h$.

Once we have the local high-order polynomial approximations of the
height function around the vertex $\mathbf{v} \in \Sigma$ and the
smooth function $\phi$ on $\Sigma$, we can also get high-order
approximation of the normal vectors, curvatures at the vertex
$\mathbf{v}$ and the gradient, Laplacian of $\phi$.

As we just discussed above, the regular surface $\Sigma$  is locally
a graph surface around the vertex  $\mathbf{v} \in \Sigma$. That is,
we can find locally a smooth height function of two variables
$z=h(x,y)$, $(x,y) \in \Omega \subset \mathbb{R}^2$ so that locally
we have the associated graph surface $\Sigma = \{ (x,y,h(x,y)) |
(x,y) \in \Omega \}$ around the vertex $\mathbf{v}$. The local graph
surface $\Sigma$ has a natural parametrization:

\begin{equation}
\mathbf{X}(u,v) = (u,v,h(u,v)), \mbox{~~}(u,v) \in \Omega.
\end{equation}
Hence, we have the tangent vectors
\begin{equation}
\mathbf{X}_u = (1, 0, h_u) \mbox{ and } \mathbf{X}_v = (0, 1, h_v)
\end{equation}
and their derivative
\begin{equation}
\mathbf{X}_{uu} = (0, 0, h_{uu}), \mbox{~} \mathbf{X}_{uv} = (0, 0, h_{uv}) \mbox{ and } \mathbf{X}_{vv} = (0, 0, h_{vv})
\end{equation}
This gives the unit normal vector

\begin{equation}
\mathbf{N}(u,v) = \frac{(-h_u,-h_v,1)}{\sqrt{h_u^2 + h_v^2+1}}, \mbox{~} (u,v) \in \Omega.
\end{equation}
The coefficients of the first fundamental form $E, F, G$ of the
graph surface $\Sigma$ are given by

\begin{equation}
E = 1+h_u^2, \mbox{~} F=h_uh_v, \mbox{~} G = 1+h_v^2
\end{equation}
and hence we have $EG-F^2 = 1+h_u^2+h_v^2$. And, the coefficients
$e, f, g$ of the second fundamental form of  $\Sigma$ are given by

\begin{equation}
e = \frac{h_{uu}}{\sqrt{1+h_u^2+h_v^2}}, \mbox{~} f = \frac{h_{uv}}{\sqrt{1+h_u^2+h_v^2}} \mbox{ and } g = \frac{h_{vv}}{\sqrt{1+h_u^2+h_v^2}}.
\end{equation}
The formulas for Gaussian and mean curvatures are given by

\begin{equation}
K = \frac{h_{uu}h_{vv} - h_{uv}^2}{(1+h_u^2+h_v^2)^2} \mbox{ and } H = \frac{1}{2}\frac{(1+h_u^2)h_{vv} - 2h_uh_vh_uv + (1+h_v)^2h_{uu}}
{(1+h_u^2+h_v^2)^{3/2}}
\end{equation}

Consider a smooth function $\phi$  on the graph surface $\Sigma$.
The gradient $\nabla _{\Sigma} \phi$  of  $\phi$ can be computed
from equation (\ref{gradient_X}) and one yields

\begin{equation}
\nabla _{\Sigma} \phi = \frac{(\phi_u(1+h_u^2)-\phi_v h_u h_v, \phi_v(1+h_v)^2-\phi_u h_u h_v, \phi_u + \phi_v)}
{1+h_u^2+h_v^2}
\end{equation}
where $\phi_u = \frac{\partial \phi (\mathbf{X}(u,v))}{\partial u}$
and  $\phi_v = \frac{\partial \phi (\mathbf{X}(u,v))}{\partial v}$.
Similarly, we can use Equations (\ref{gradient_X})-(\ref{laplace_g})
to compute the divergence of a vector field $\mathbf{V}$ and the
Laplacian $\Delta _{\Sigma} \phi$ of a smooth function $\phi$ on the
graph surface $\Sigma$.

From these equations, we can conclude that if we have  $k$th-order
polynomial approximations of the height function $h$ around the
vertex $\mathbf{v} \in \Sigma$ and the smooth function $\phi$  on
$\Sigma$, we obtain $(k-1)$th-order approximations for the normal
vector, $E, F ,G$ and the gradient of $\phi$. We also obtain
$(k-2)$th-order approximations for the Gaussian  and mean curvatures
and also the Laplacian of $\phi$  on $\Sigma$.

\section{Numerical simulations}\label{simulation}
In this section,  we present some numerical simulations of our
proposed methods in sections 3 and 4. The models in our simulations
are the unit sphere, the torus with inner radius $0.5$ and outer
radius $1$, a dumbbell with the parametrization
\begin{equation}\label{eqn_dumbbell}
\mathbf{X}(u,v) = \begin{pmatrix} r \sin v \cos u, & r \sin v \sin u, & r \cos v \end{pmatrix}
\end{equation}
where $r=\sqrt{0.9^2 \cos (2x) + \sqrt{1 - 0.9^4 sin^2 (2x)}}$, see
Figure \ref{fig_dumbbell}, and a wave surface with
\begin{equation}\label{eqn_parametricsurface}
\mathbf{X}(u,v) = \begin{pmatrix}u, v, \sin u \cos(v) \end{pmatrix},
\end{equation}  where $(u,v) \in [0, 2\pi]\times [0, 2\pi]$, see Figure \ref{fig_parametric_surface}.

\begin{figure}[!t]
\centering
\includegraphics[height=2in, bb=  0 0 468 714]{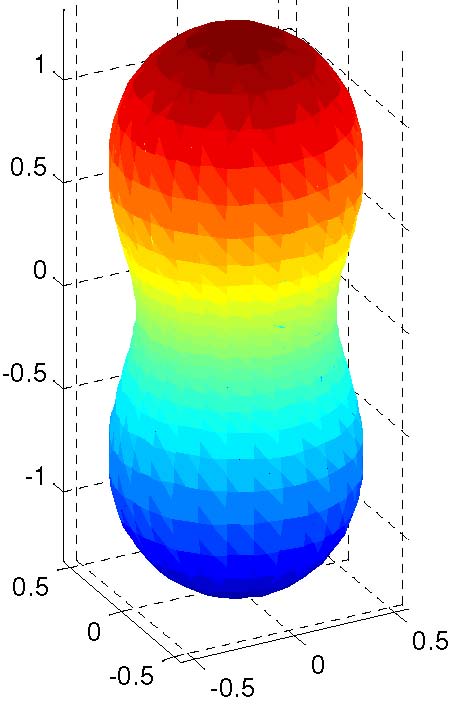}
\caption{The figure of a dumbbell}\label{fig_dumbbell}

\end{figure}

\begin{figure}[!t]
\centering
\includegraphics[scale=.5]{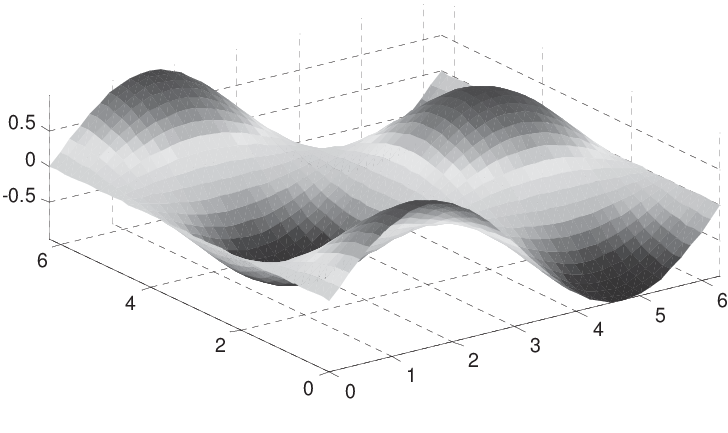}
\caption{The figure of the wave surface}\label{fig_parametric_surface}

\end{figure}

\subsection{Simulations for the $O(r)$-LTL algorithm }
First, we show the numerical solutions of eigenvalues of
Laplace-Beltrami operator on the unit sphere and some curved
surfaces by the $O(r)$-LTL method in section
\ref{section_configuration_LB_1}.
 The $O(r)$-LTL approach  is an special case for our
$O(r^n)$-LTL method. Although the $O(r)$-LTL method has a lower
convergent rate,  it only requires at least 5 neighboring vertices
of each vertex on the surfaces and the triangular meshes always can
be reconstructed such that the number of $1$-ring neighboring
vertices  of each vertex is at least $5$.

Table \ref{eigenvalues_sphere} shows the simulation results of the
eigenvalues and its multiplicity of the Laplace-Beltrami operator on
the unit sphere by our $O(r)$-LTL method. The triangular mesh sizes
of the unit sphere in our simulations are $0.32, 0.16, 0.08$ and
$0.04$. We give these triangular meshes in Figure
\ref{sphere_model}. Figure \ref{eig_torus} gives some eigenfunctions
on the torus with inner radius $0.5$ and outer radius $1$ and Figure
\ref{eig_sphere} shows all eigenfunctions corresponding to the first
5 eigenvalues on a unit sphere by the $O(r)$-LTL method.

Figure \ref{eig_hemisphere} shows the convergence result of the
eigenvalues of a unit hemisphere by our $O(r)$-LTL  method. Our
$O(r)$-LTL  method has the  quadratic convergence.

\begin{figure}[!t]
\centering
\includegraphics[width=2.5in]{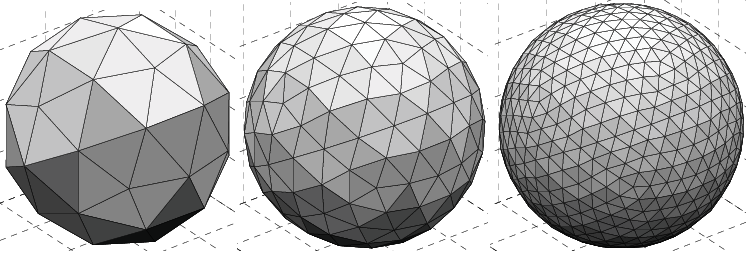}
\caption{The subdivision of triangular meshes of a unit sphere by $O(r)$-LTL method}\label{sphere_model}

\end{figure}

\begin{table}[!t]
\renewcommand{\arraystretch}{1.3}
\caption{The eigenvalues and its multiplicity of a unit sphere}
\label{eigenvalues_sphere} \centering
\begin{tabular}{|c||c||c||c|}
\hline
mesh size & 1st eigenvalue& 2nd eigenvalue  & 3th eigenvalue  \\
  &  (multiplicity) & (multiplicity) & (multiplicity) \\
\hline
0.32 & -2.0484 (3) & -6.0014 (5) & -11.5986 (7) \\
0.16 &  -2.0117 (3) & -6.0002 (5) & -11,8998 (7)  \\
0.08 &  -2.0029 (3) & -6.0 (5) & -11.9735 (7) \\
0.04 & -2.0 (3)  & -6.0 (5) & -12.0 (7) \\
\hline
\end{tabular}
\end{table}

\begin{figure}[!t]
\centering
\includegraphics[width=2in, bb=0 0 1693 1212]{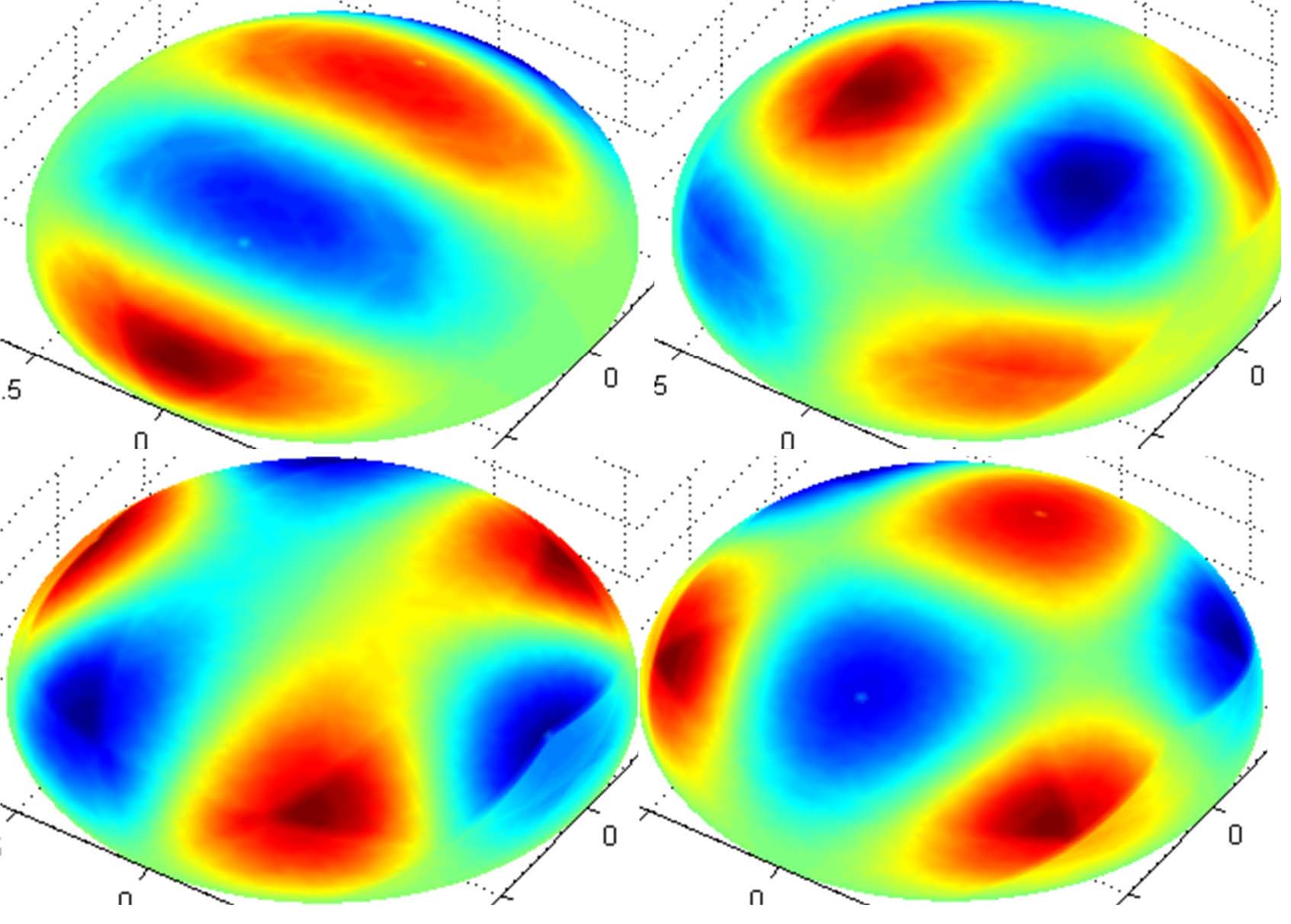}
\includegraphics[width=2.5in]{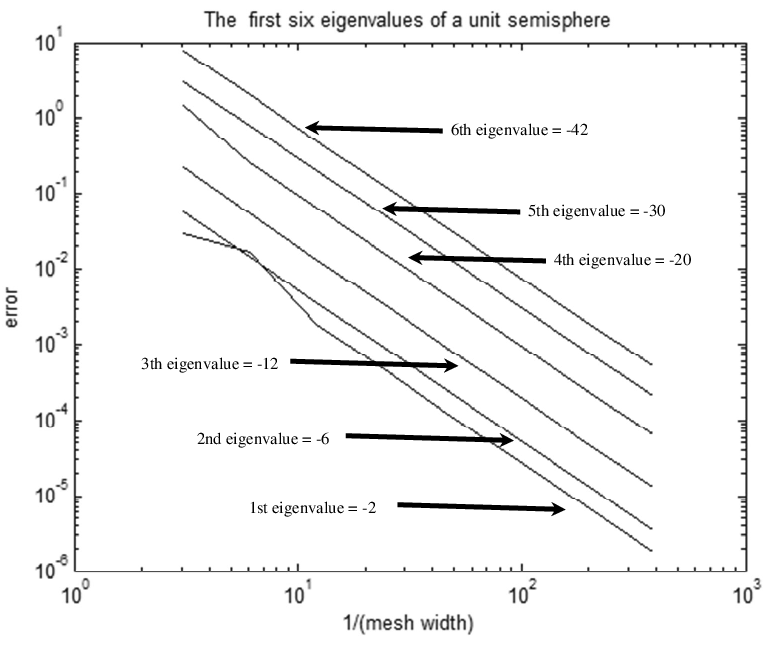}
\caption{The left  shows some eigenfunctions of a hemisphere by $O(r)$-LTL method.
The right  is a convergence study for the first six eigenvalues of  a unit hemisphere.}\label{eig_hemisphere}
\end{figure}

\begin{figure}[!t]
\centering
\includegraphics[ width=.8\textwidth, bb=0 0 1255 501]{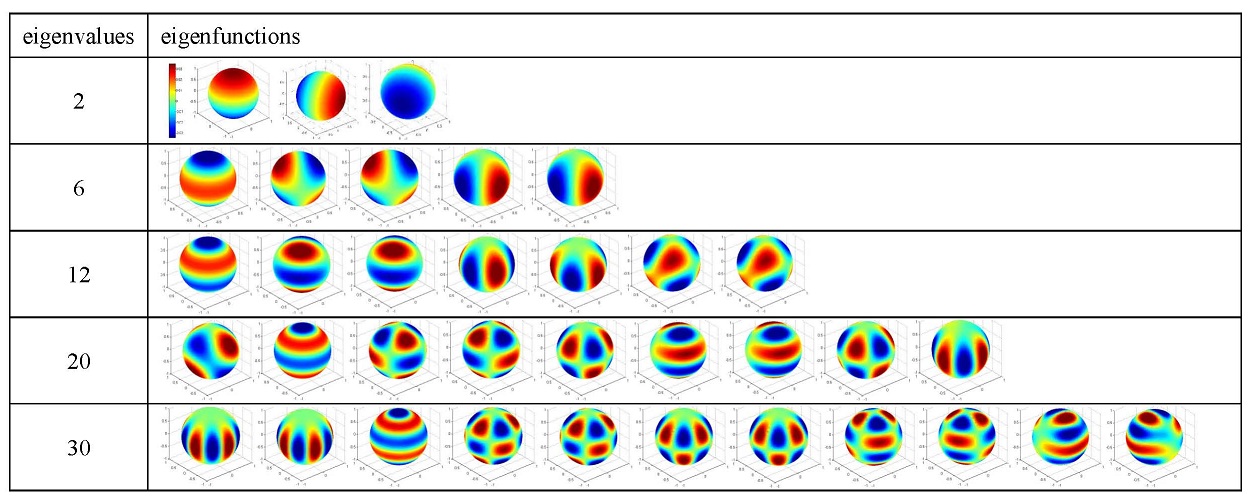}
\caption{The eigenfunctions of the Laplace-Beltrami operator on a unit sphere by $O(r)$-LTL method.}\label{eig_sphere}
\end{figure}

\begin{figure}[!t]
\centering
\includegraphics[width=3in,bb = 0 0 1428 1139]{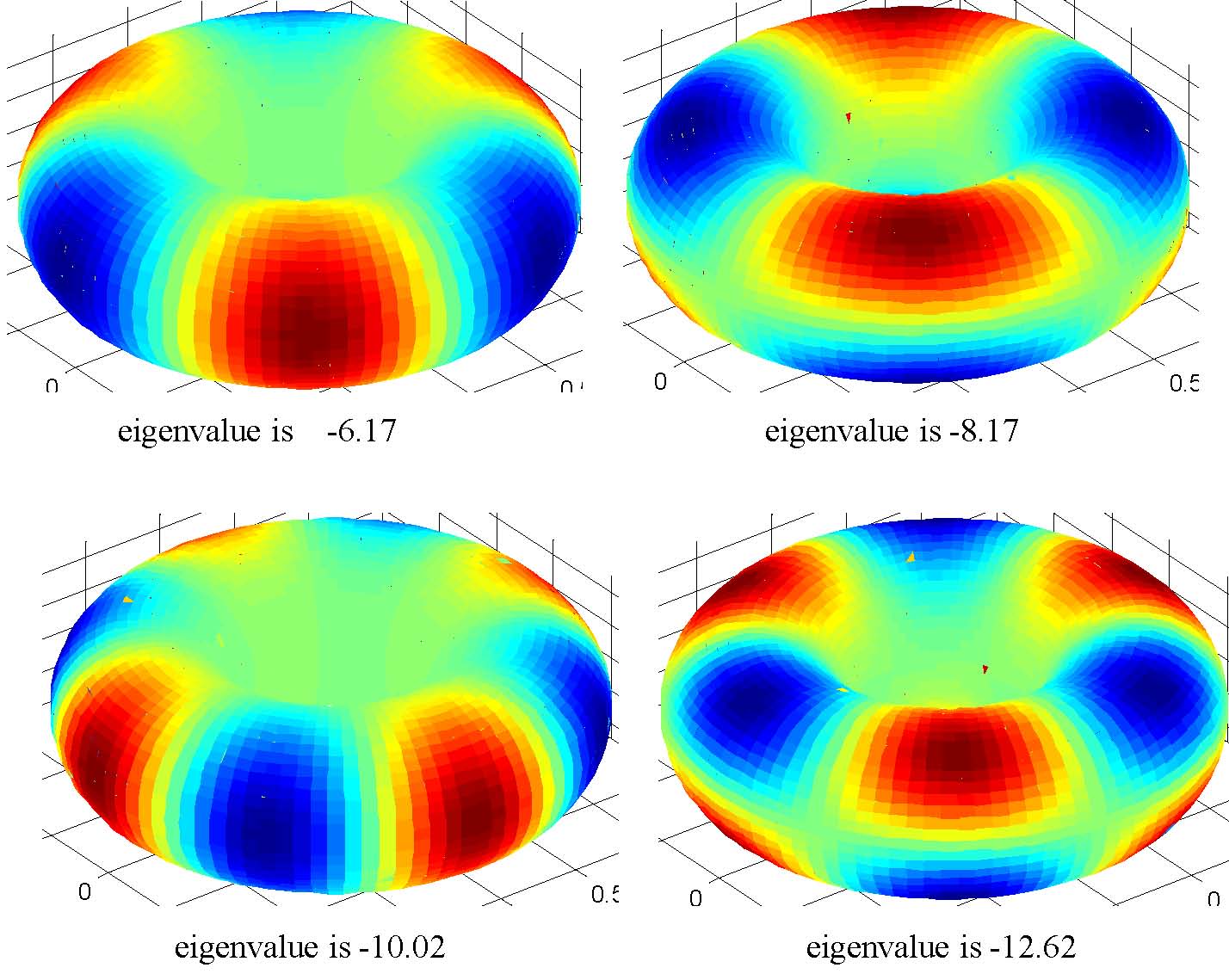}
\caption{The eigenfunctions of the Laplace-Beltrami operator on a torus by $O(r)$-LTL method.}\label{eig_torus}
\end{figure}

\subsection{ Simulations for the $O(r^n$)-LTL algorithm with $n \geq 2$ }

In this subsection, we compute the Laplacian of a   function,
\begin{equation}
F(u,v) = \exp( 0.5 \sin u + \cos ^3 v )
\end{equation}
where $(u,v)$ defined on the domain of $\mathbf{X}(u,v)$ in Equation
(\ref{eqn_parametricsurface}), on the parametric surface in Equation
(\ref{eqn_parametricsurface}) by high-order algorithms. How to
determine the necessary neighboring vertices of a given vertex is an
important problem for high-order accuracies. Ray et al.\cite{Ray}
proposed an elegant method to improve this problem by using the
$\frac{j}{2}$-ring ($j \geq 2$) for neighboring vertices, see Figure
\ref{fig_neighbor}. We will also use the  $\frac{j}{2}$-ring
neighboring vertices in our simulations.

We estimate the $l^\infty$ relative error of Laplacian on each mesh
as
\begin{equation}
l^{\infty} \mbox{ relative error} = \max_{ \mathbf{v}_i \in V} ( \mbox{ relative error at } \mathbf{v}_i )
\end{equation}
where
\begin{equation}
\mbox{ relative error at } \mathbf{v}_i = \frac{ \| \mbox{ numerical solution at }\mathbf{v}_i -
\mbox{ reference solution at } \mathbf{v}_i\|}{\|\mbox{ reference solution at } \mathbf{v}_i\|}
\end{equation}
and $\|\mbox{ reference solution at } \mathbf{v}_i\| \neq 0$.

 Figure \ref{error_laplacian_parametric_surface} shows the $l^\infty$
relative errors of Laplacian of the function $F(u,v) = \exp(0.5 \sin
u + \cos ^3 v)$ in the interior vertices on this parametric surface.
These results always converge to the exact Laplacian when the mesh
size of the triangular mesh approaches $0$.

\begin{figure}[!t]
\centering
\includegraphics[width=3in]{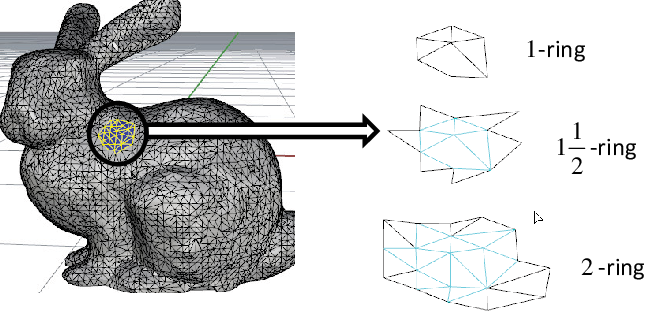}
\caption{ The $1$, $1\frac{1}{2}$ and $2$-rings of a vertex on the bunny model.}\label{fig_neighbor}
\end{figure}

\begin{figure}[!t]
\centering
\includegraphics[width=3in]{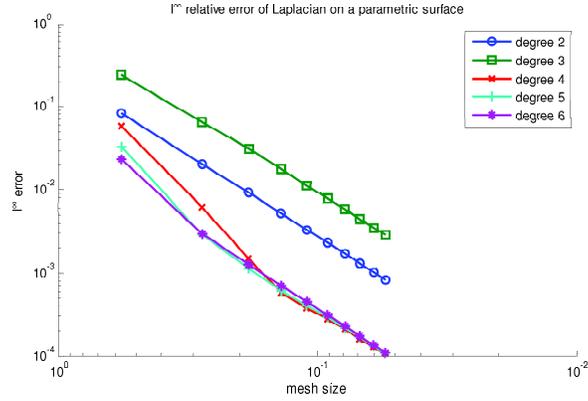}
\caption{The $l^\infty$ relative errors of laplacian of the function
$F(u,v)=\exp(0.5 \sin u + \cos ^2 v)$ in Equation  on a wave surface in Equation (\ref{eqn_parametricsurface}) }\label{error_laplacian_parametric_surface}

\end{figure}

Now, we compare the eigenvalues and eigenfunctions on a unit sphere.
It is well-known that the $n$-th nonzero eigenvalues of the
Laplace-Beltrami operator on the unit sphere are $-n(n+1)$ with
multiplicity $(2n+1)$. The eigenfunctions of the unit sphere are
restrictions on the unit sphere of harmonic homogeneous polynomials
in $\mathbb{R}^3$. Figures \ref{er_1st_eigenvalue_sphere} and
\ref{er_2nd_eigenvalue_sphere} show the  errors of the eigenvalues
$\lambda=2$ and $6$ on a unit sphere, repsectively. For the error
$e(h_1)$ and $e(h_2)$ for the mesh size $h_1$ and $h_2$ the
experimental error of convergence is defined as
\begin{equation}
EOC(h_1,h_2) = \log\frac{e(h_1) }{e(h_2)}\left ( \log \frac{h_1}{h_2} \right ) ^{-1}.
\end{equation}
Tables \ref{EOC_eigenval1} and \ref{EOC_eigenval2} show the EOC of
the first and second eigenvalues on the unit sphere, respectively.

\begin{table}[!t]
\centering \caption{EOC of the 1st eigenvalue on a unit
sphere}\label{EOC_eigenval1}
\begin{tabular}{cccccc}
\hline  \\
mesh size & degree 2 & degree 3 & degree 4 & degree 5 & degree 6      \\
\hline
0.5443   & - & - & - & - & - \\
0.2774 & 2.0971 & 1.9543 & 3.6868 & 4.1610 & 6.7939    \\
0.1877 & 2.0901 & 1.9751 & 4.1204 & 3.8777 & 5.9443    \\
0.1407 & 2.0049 & 1.9327 & 4.0126 & 3.8251 & 6.1007     \\
0.1130 & 2.0384 & 1.9872 & 4.0847 & 3.9515 & 6.2562    \\
0.0941 & 1.9998 & 1.9626 & 4.0049 & 3.9122 & 5.9561   \\
0.0807 & 2.0243 & 1.9949 & 4.0537 & 3.9847 & 6.9588   \\
0.0706 & 1.9989 & 1.9754 & 3.9973 & 3.9468 & 5.8779  \\
\hline
\end{tabular}
\end{table}

\begin{table}[!t]
\centering \caption{EOC of the 2nd eigenvalue on a unit
sphere}\label{EOC_eigenval2}
\begin{tabular}{cccccc}
\hline  \\
mesh size & degree 2 & degree 3 & degree 4 & degree 5 & degree 6      \\
\hline
0.5443   & - & - & - & - & - \\
0.2774 & 1.8991     & 3.4314    &  2.8961  &    5.3251  &    5.9809  \\
0.1877 & 2.0115     & 4.0060    &  3.7518  &    5.8448  &    5.2146  \\
0.1407 & 1.9552     & 3.9231    &  3.7957  &    5.3624  &    3.8108  \\
0.1130 & 1.9999     & 4.0148    &  3.9412  &    5.2539  &    3.8859   \\
0.0941 & 1.9691     & 3.9537    &  3.9086  &    4.9920  &    3.7684   \\
0.0807 & 1.9981     & 4.0113    &  3.9810  &    4.8949  &    3.7388   \\
0.0706 & 1.9765     & 3.9473    &  3.9413  &    5.1215  &    3.6412   \\
\hline
\end{tabular}
\end{table}

Let $\{ \phi_{n,i} \}_{i=1}^{2n+1}$ be the set of eigenfunctions
corresponding to the eigenvalue $-n(n+1)$ on the unit sphere and let
$\tilde{\phi}_{n,i}$ be the approximated solution by our
$O(r^n)$-LTL method. Then we choose $\bar{\phi}_{n,i}$ to be the
linear combination $\sum_{i=1}^{2n+1} \alpha_i \phi_{n,i}$ so that
it realizes the minimum of $\| \sum_{i=1}^{2n+1} \beta_j \phi_{n,i}
- \tilde{\phi}_{n,i}\|$ over $\beta_j$'s. The error $E_n$ of the nth
eigenvalue with eigenfunctions $\bar{\phi}_{n,i}$ and the
approximating eigenfunctions $\tilde{\phi}_{n,i}$ is defined by
\begin{equation}\label{l_infty_infty_error}
 E_n = \sup_{i \in \{1,2,\cdots, 2n+1\}} (\|\bar{\phi}_{n,i} - \tilde{\phi}_{n,i} \|_{l^\infty}).
\end{equation}
Figures \ref{er_1st_eigenfunction_sphere} and
\ref{er_2nd_eigenfunction_sphere} show  the   errors $E_n$ in
Equation (\ref{l_infty_infty_error}) of the 1st and 2nd eigenvalues
on a unit sphere, respectively.

\begin{figure}[!t]
\centering
\includegraphics[width=3in]{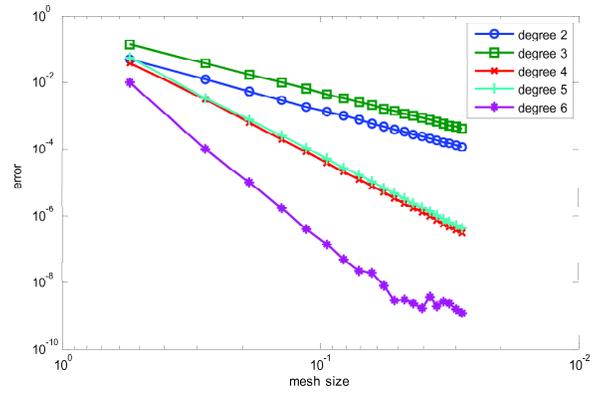}
\caption{ The  errors of the 1st eigenvalues on a unit sphere. }\label{er_1st_eigenvalue_sphere}
\end{figure}

\begin{figure}[!t]
\centering
\includegraphics[width=3in]{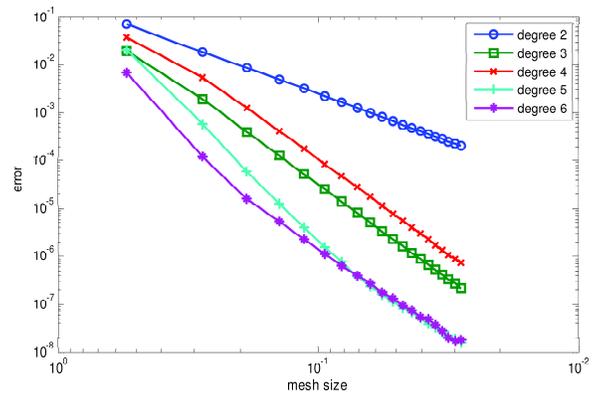}
\caption{ The $l^\infty$  errors of the 2nd eigenvalues   on a unit sphere. }\label{er_2nd_eigenvalue_sphere}
\end{figure}

\begin{figure}[!t]
\centering
\includegraphics[width=3in]{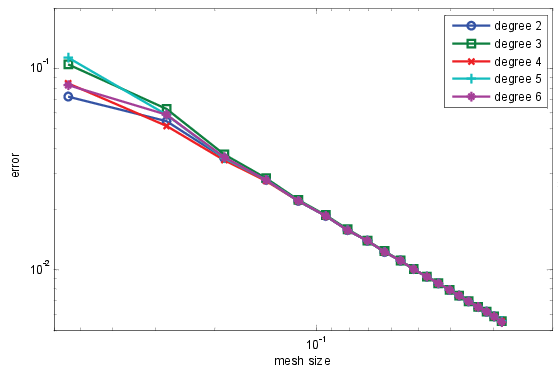}
\caption{ $E_1$   on a unit sphere. }\label{er_1st_eigenfunction_sphere}
\end{figure}

\begin{figure}[!t]
\centering
\includegraphics[width=3in]{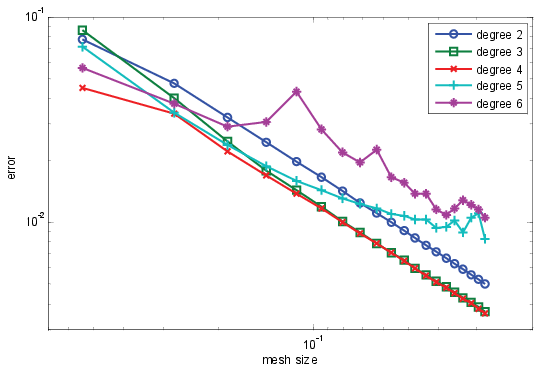}
\caption{ $E_2$ on a unit sphere. }\label{er_2nd_eigenfunction_sphere}
\end{figure}

\subsection{Geometric invariants on curved surface}
Finally, we estimate the normal vectors and the tensor of curvatures
on regular surfaces. Figures \ref{error_normals_torus} -
\ref{error_mean_curvature_torus} show  the $\l^\infty$ relative
errors of normal vectors, Gaussian curvatures and mean curvatures on
a torus with the inner radius $0.5$ and the outer radius $1$.
Figures \ref{error_gaussian_curvature_dumbbell} -
\ref{error_mean_curvature_dumbbell} give the $\l^\infty$ relative
errors of normal vectors, Gaussian curvatures and mean curvatures on
a dumbbell  in Equation (\ref{eqn_dumbbell}) and Figure
\ref{error_gaussian_curvature_wavesurface} shows the $\l^\infty$
errors of Gaussian curvatures on a wave surface. Obviously, the
high-order algorithm is much more accurate than
 the low-order algorithm when the mesh size is small
enough.

\begin{figure}[!t]
\centering
\includegraphics[width=3in]{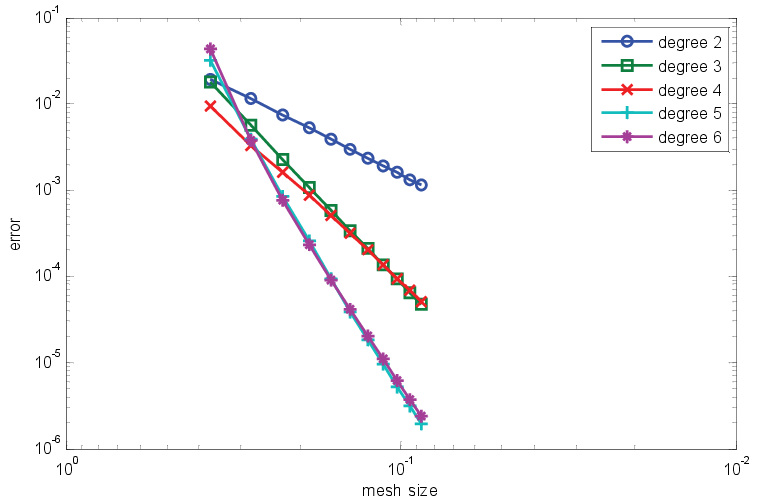}
\caption{The $l^\infty$ relative errors of normal vectors on a torus }\label{error_normals_torus}
\end{figure}

\begin{figure}[!t]
\centering
\includegraphics[width=3in]{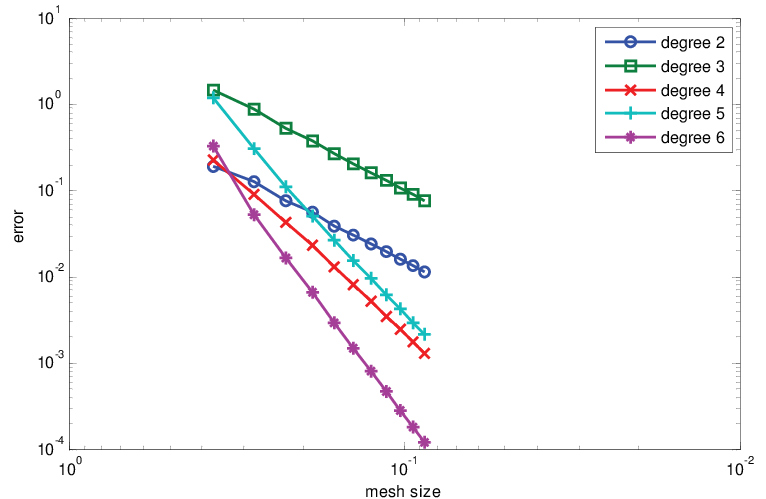}
\caption{The $l^\infty$ relative error of Gaussian curvatures on a torus }\label{error_gaussian_curvature_torus}

\end{figure}

\begin{figure}[!t]
\centering
\includegraphics[width=3in]{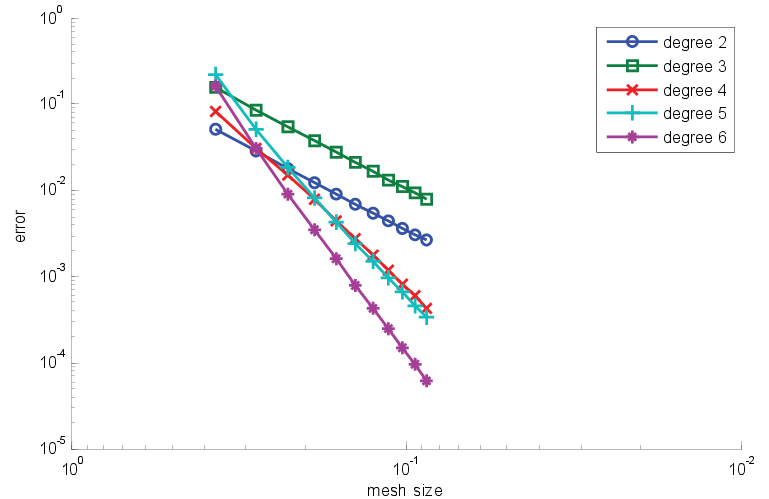}
\caption{The $l^\infty$ relative error of mean curvatures on a torus }\label{error_mean_curvature_torus}

\end{figure}

\begin{figure}[!t]
\centering
\includegraphics[width= 3in]{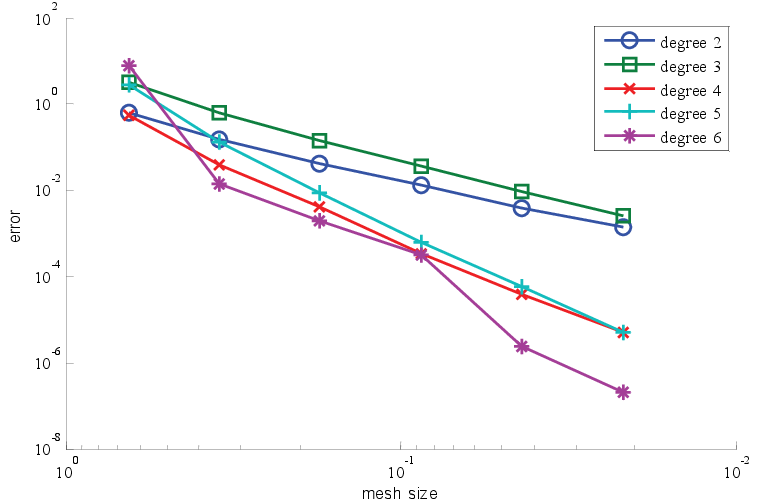}
\caption{The $l^\infty$ relative errors of Gaussian curvatures on a dumbbell}\label{error_gaussian_curvature_dumbbell}

\end{figure}

\begin{figure}[!t]
\centering
\includegraphics[width= 3in]{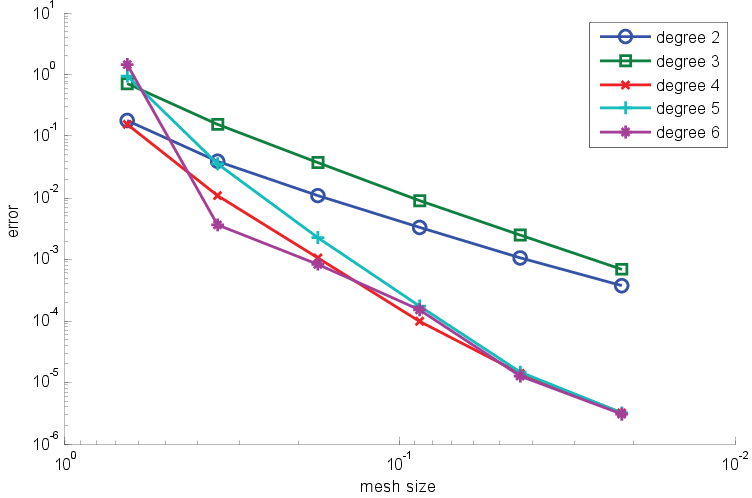}
\caption{The $l^\infty$ relative error of mean curvatures on a dumbbell}\label{error_mean_curvature_dumbbell}

\end{figure}

\begin{figure}[!t]
\centering
\includegraphics[width= 3in]{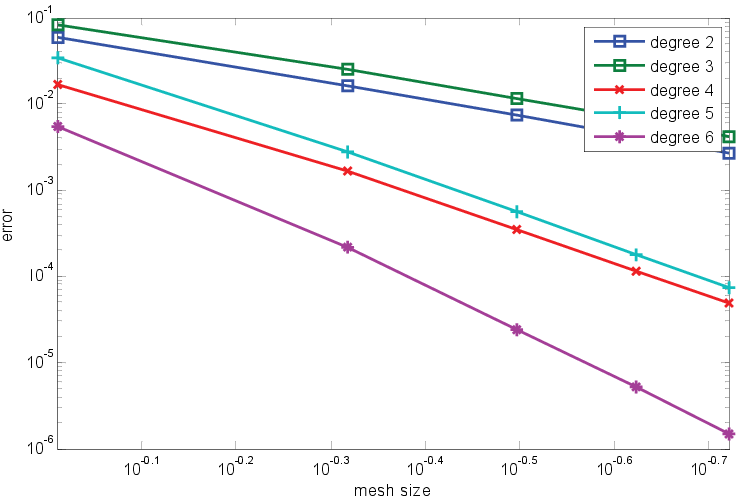}
\caption{The $l^\infty$ relative error of Gaussian curvatures on a wave surface}\label{error_gaussian_curvature_wavesurface}

\end{figure}

\section{Dissusions and Conclusions}

In 2012, Ray et al. also proposed  a high-order numerical method for
estimating the derivatives and integrations over discrete surfaces.
Ray et al. solved the system of linear equations from Taylor
expansion by the least square method. However,  we deal with these
problems from the viewpoint of duality and obtain  our configuration
equations for the  Laplace-Beltrami operators over discrete
surfaces.  The  configuration equation for Laplace-Beltrami operator
is

\begin{equation}\label{configuration_LB_new}
\begin{pmatrix}
x_1, & x_2, & \cdots, & x_n \cr
y_1, & y_2, & \cdots, & y_n \cr
\frac{x_1^2}{2}, & \frac{x_2^2}{2}, & \cdots, & \frac{x_n^2}{2} \cr
x_1 y_1 & x_2 y_2, & \cdots, & x_n y_n \cr
\frac{y_1^2}{2}, & \frac{y_2^2}{2}, & \cdots, & \frac{y_n^2}{2}
\end{pmatrix}
\begin{pmatrix}
\alpha_1 \cr
\alpha_2 \cr
\vdots \cr
\alpha_n
\end{pmatrix}
= \begin{pmatrix} 0 \cr 0 \cr 1 \cr 0 \cr 1 \end{pmatrix}
\end{equation}
and the Laplacian $\Delta f(0,0)$ in Equation (\ref{laplace_1}) is
given by
\begin{equation}
\Delta f(0,0) = \sum_{i=1}^n \alpha_i(f(x_i,y_i) - f(0,0)).
\end{equation}
Similarly, we can also obtain a high-order configuration equation
for the Laplace-Beltrami operator in the $O(r^n)$-LTL method.

Although the solution of the configuration equation  may not be
unique in general, the convergence rates of all solutions are
similar by the theoretical analysis. However, not every solution has
a good numerical  simulation. Finding a suitable solution $\{
\alpha_i \}_{i=1}^n$ is a key problem in Ray's method and our
$O(r^n)$-LTL methods. In Ray's method,   this problem can be
improved   by using the conditional number\cite{Ray}. The pseudo
inverse of the configuration matrix is a good approach to handle
this problem in our $O(r^n)$-LTL methods. Hence, we have
\begin{equation}
\begin{pmatrix}
\alpha_1 \cr
\alpha_2 \cr
\vdots \cr
\alpha_n
\end{pmatrix}
= \mathrm{pinv}( \begin{pmatrix} x_1, & x_2, & \cdots, & x_n \cr
y_1, & y_2, &  \cdots, &  y_n \cr x_1 y_1, &  x_2 y_2, & \cdots, & x_n y_n \cr
x_1^2 - y_1^2, & x_2^2 y_2^2, & \cdots, &  x_n^2 - y_n^2 \cr 1, & 1,
& \cdots, & 1 \end{pmatrix}) \begin{pmatrix} 0 \cr 0 \cr 0 \cr 0
\cr \epsilon \end{pmatrix}
\end{equation}
in Equation (\ref{configuration_LB}) where
$\mathrm{pinv}(\mathbf{M})$ is the pseudo inverse of the matrix
$\mathbf{M}$.

As expected, the high-order approach is more accurate than the
low-order approach. However,  the low-order approach is more stable
than the high-order approach in our simulations. When we use the
high-order approach, we need more neighboring vertices at each
vertex and the structure of $\frac{j}{2}$-ring becomes more
complicated. For instance we need at least 26 neighboring vertices
for the degree $6$ approximations.

Note that our $O(r)$-LTL method is a special case of the
$O(r^n)$-LTL methods. In the $O(r)$-LTL method, we only need $5$
neighboring vertices at each vertex for estimating the Laplacian of
a function on a surface by the configuration equation
(\ref{configuration_LB}). Almost all vertices on a closed triangular
mesh can be reconstructed so that the number of $1$-ring is at least
$5$ for each vertex in the new mesh.  Ray's method and our
high-order configuration equation, Equation
(\ref{configuration_LB_new}), need at least $6$ neighboring
vertices. Indeed, our $O(r)$-LTL method for estimating Laplacian is
a generalization of the well-known the $5$- point Laplacian method
in $\mathbb{R}^2$.

In our $O(r)$-LTL configuration method,  Equations
(\ref{problem_LB_1}) and (2) can be reduced to the matrix equations
$$ (WA)(\phi(V)) = \lambda(\phi(V)) $$
and
$$ (WB)(\phi(V)) = \lambda(\phi(V)), $$
respectively. The configuration masks $A=(a_{i,k}), B=(b_{i,k})$ are
given by
\begin{equation}
 a_{i,k} = \left \{ \begin{array}{ll} \alpha_{i,j} & \mbox{ if } v_k = v_{i,j}, \cr
 -1 & \mbox{ if } i=k, \cr
 0 & \mbox{ otherwise,} \end{array} \right.
\end{equation}
and
\begin{equation}
 b_{i,k} = \left \{ \begin{array}{ll} \beta_{i,j} & \mbox{ if } v_k = v_{i,j}, \cr
 -1 & \mbox{ if } i=k, \cr
 0 & \mbox{ otherwise,} \end{array} \right.
\end{equation}
where $\alpha_{i,j}$ is a solution of Equation
(\ref{configuration_LB}) and $\beta_{i,j}$ is a solution of Equation
(\ref{configuration_LB2}). Using the configuration matrix $A$, we
can also solve the diffusion equation
$$ u_t - \Delta_{\Sigma} u = f$$
 on regular surface $\Sigma$ easily. Furthermore, our $O(r^n)$-LTL configuration method can also solve
 general partial differential equations, $L(u) = f$, on a regular surface $\Sigma$ or on more complicated domains,
 like Koch snowflakes. The key of our $O(r^n)$-LTL configuration
 method for solving the PDE $L(u) = f$ is to find the configuration matrix of the differential operator $L$
 on the surface $\Sigma$.

It is worth to point out that our $O(r^n)$-LTL configuration method
is effective to solve the eigenpair problems with high-order
accuracies even when the underlying surfaces or domains have
complicated topological or geometrical structures.

In the near future, we shall extend our $O(r^n)$-LTL configuration
method to solve more partial differential equations on regular
surfaces.

\section*{Acknowledgment}
This paper is partially supported by NSC, Taiwan.

\end{document}